\documentclass[12pt]{article} 

\usepackage[margin=1in]{geometry}
\usepackage{amsfonts,latexsym,amsthm,amssymb,amsmath,amscd,euscript,esint, dsfont}	
\usepackage{graphicx}		
\usepackage{hyperref}
\usepackage{cases}			
\usepackage{textcomp, gensymb}
\usepackage{xcolor}
\usepackage{tabularx}

\usepackage{quiver,tikz}

\usepackage[margin=10pt,font=small,labelfont=bf]{caption}

\makeatletter
\def\namedlabel#1#2{\begingroup
	\def\@currentlabel{#2}
	\label{#1}\endgroup
}
\makeatother

\usepackage[parfill]{parskip}
\usepackage{lastpage}
\usepackage{fancyhdr}

\newcommand{\EE}{\mathbb{E}}

\newcommand{\HH}{\mathbb{H}}

\newcommand{\QQ}{\mathbb{Q}}
\newcommand{\RR}{\mathbb{R}}

\newcommand{\TT}{\mathbb{T}}

\newcommand{\ZZ}{\mathbb{Z}}

\newcommand{\cS}{\mathcal{S}}
\newcommand{\cT}{\mathcal{T}}

\newcommand{\ks}{\mathfrak{s}}

\newcommand{\Deck}{\operatorname{Deck}}

\newcommand{\Id}{\operatorname{Id}}

\newcommand{\Aut}{\operatorname{Aut}}

\newcommand{\Out}{\operatorname{Out}}

\newcommand{\GL}{\operatorname{GL}}

\newcommand{\SO}{\operatorname{SO}}

\newcommand{\Diff}{\operatorname{Diff}}
\newcommand{\Isom}{\operatorname{Isom}}

\newcommand{\orb}{\text{orb}}

\newcommand{\Orb}{\operatorname{Orb}}

\newcommand{\ch}{\operatorname{ch}}

\newtheorem{theorem}{Theorem}[section]
\newtheorem{corollary}[theorem]{Corollary}
\newtheorem{lemma}[theorem]{Lemma}

\newtheorem{proposition}[theorem]{Proposition}

\newtheorem{conjecture}[theorem]{Conjecture}

\theoremstyle{definition}

\newtheorem{example}[theorem]{Example}
\newtheorem{remark}[theorem]{Remark}
\newtheorem{problem}[theorem]{Problem}

\usepackage{hyperref}
\hypersetup{
    colorlinks=true,
    linkcolor=blue,
    filecolor=magenta,      
    urlcolor=cyan,
    pdftitle={Overleaf Example},
    pdfpagemode=FullScreen,
}

\usepackage[shortlabels]{enumitem}

\usepackage[
backend=bibtex,
style=alphabetic,
sorting=nyt
]{biblatex}
\graphicspath{ {./figure/} }

\title{Isometric free finite group actions on non-positively curved 3-manifolds}
\author{Zhengyu Zou}
\date{July 2026}

\begin{document}

\maketitle

\begin{abstract}
    Let $M$ be a closed orientable $3$-manifold admitting a metric of nonpositive sectional curvature (an NPC metric), and let $G$ be a finite group acting freely on $M$ by orientation-preserving diffeomorphisms. Previous results showed that $M$ admits a $G$-invariant NPC metric except possibly when $M$ is a graph manifold. In this paper, we resolve the remaining case by proving that $M$ also admits a $G$-invariant NPC metric when $M$ is a graph manifold. This result advances our understanding in dimension $3$ of the question posed by Schoen--Yau about Nielsen realization for NPC $3$-manifolds.
\end{abstract}

\section{Introduction} \label{sec:intro}

Riemannian metrics of nonpositive sectional curvature (NPC metrics) are rigid enough to impose strong topological restrictions on $3$-manifolds, yet flexible enough to occur on many closed $3$-manifolds beyond the locally homogeneous setting. Schoen--Yau \cite{schoen-yau-1979-compact-group-actions} proposed the following analog of the Nielsen realization problem for manifolds admitting NPC metrics.

\medskip

\begin{problem}[Schoen--Yau] \label{prob:Schoen-Yau}
    Let $M$ be a compact manifold admitting an NPC metric, and let $G$ be a finite subgroup of $\Out(\pi_1(M))$. Does there exist an NPC metric on $M$ such that $G$ is realized by a finite subgroup of $\Isom(M)$?
\end{problem}

There has been partial progress on Problem~\ref{prob:Schoen-Yau}. In dimension $2$, it is answered by Kerckhoff's solution to the classical Nielsen realization problem \cite{Kerckhoff1983Nielsen}. When $M$ is hyperbolic and $\dim M \geq 3$, Problem~\ref{prob:Schoen-Yau} is answered by Mostow rigidity. On the other hand, in dimension $n \geq 7$, Farrell--Jones \cite[Theorem~1]{FarrellJones1990} constructed exotic smooth structures on closed manifolds $M^n$ that are homeomorphic to hyperbolic manifolds. For these exotic smooth structures, the natural map $\Diff(M^n) \rightarrow \Out(\pi_1(M^n))$ is not surjective, even though $M^n$ admits metrics of negative sectional curvature. Thus Problem~\ref{prob:Schoen-Yau} is false in dimensions at least $7$. 

Our work is motivated by the dimension-$3$ case of Problem~\ref{prob:Schoen-Yau}. Earlier work of Zimmermann showed that, when $M^3$ is Haken and not Seifert fibered, every finite subgroup of $\Out(\pi_1(M))$ can be lifted to $\Diff(M)$ \cite[Satz~0.1]{Zimmermann1982}. However, when $M$ is Seifert fibered, Heil--Tollefson found counterexamples to Problem~\ref{prob:Schoen-Yau} by analyzing obstructions to realizing extensions of finite subgroups of $\Out(\pi_1(M))$ by $\pi_1(M)$ \cite[Example~2.8]{HeilTollefson1978}. Now let $M$ be a closed, orientable $3$-manifold admitting an NPC metric. Then $M$ is irreducible. If the JSJ decomposition of $M$ contains more than one component, then $M$ contains an incompressible embedded torus; hence $M$ is Haken and is not Seifert fibered. In this setting, Problem~\ref{prob:Schoen-Yau} and Zimmermann's smooth realization result motivate the following geometric realization problem.

\medskip

\begin{problem}\label{prob:npc_isometric}
Let $M$ be a closed, orientable $3$-manifold admitting an NPC metric, and let $G$ be a finite group acting on $M$ by orientation-preserving diffeomorphisms. Does $M$ admit a $G$-invariant NPC metric?
\end{problem}

In this paper, we study Problem~\ref{prob:npc_isometric} under the additional assumption that $G$ acts freely on $M$. Under this assumption, the problem is already known to have a positive answer when $M$ is Seifert fibered \cite[Theorem~2.1]{MeeksScott1986FiniteActions3Manifolds}, when $M$ is atoroidal but not Seifert fibered \cite{Perelman2002,Perelman2003Surgery,Perelman2003Extinction,DinkelbachLeeb2009}, and more generally when the JSJ decomposition of $M$ contains at least one atoroidal piece \cite[Theorem~3.3]{Leeb1995NPC3Manifolds}. The remaining case is therefore the case of \emph{graph manifolds}, namely closed $3$-manifolds whose JSJ pieces are all Seifert fibered. This case is genuinely subtle: even if each individual Seifert fibered piece admits an NPC metric, the resulting graph manifold need not admit one. Our main result addresses precisely this remaining case.

\medskip

\begin{theorem}\label{thm:npc_metric_free_action_invariant}
Let $M$ be a closed, orientable graph manifold admitting an NPC metric, and let $G$ be a finite group acting freely on $M$ by orientation-preserving diffeomorphisms. Then the quotient manifold $M/G$ admits an NPC metric. In particular, $M$ admits a $G$-invariant NPC metric obtained by pulling back an NPC metric on $M/G$.
\end{theorem}

Combining Theorem~\ref{thm:npc_metric_free_action_invariant} with the previously known cases gives the following corollary.

\medskip

\begin{corollary}\label{cor:npc_metric_free_action_invariant}
Let $M$ be a closed, orientable $3$-manifold admitting an NPC metric, and let $G$ be a finite group acting freely on $M$ by orientation-preserving diffeomorphisms. Then $M$ admits a $G$-invariant NPC metric.
\end{corollary}

Problem~\ref{prob:npc_isometric} is also of independent interest. More generally, one may ask whether a finite subgroup of $\Diff(M)$ must preserve some distinguished geometric structure on $M$. In dimension $2$, every finite group action on a closed surface preserves a metric of constant curvature. In dimension $3$, Meeks--Scott \cite{MeeksScott1986FiniteActions3Manifolds,} and Dinkelbach--Leeb \cite{DinkelbachLeeb2009} showed that finite group actions on closed manifolds modeled on one of Thurston's eight geometries can likewise be realized by isometries of a suitable geometric metric. 

\subsection{Proof Roadmap}\label{subsec:roadmap}

Theorem \ref{thm:npc_metric_free_action_invariant} assumes that $G$ acts freely on $M$. Set $M' = M/G$. Then $M'$ is again a closed, orientable $3$-manifold, and it suffices to prove that $M'$ admits an NPC metric.

Let $p : M \rightarrow M'$ be the covering projection. To construct an NPC metric on $M'$, we begin with the JSJ decomposition of $M'$ \cite{JacoShalen1979,Johannson1979}, which splits $M'$ along a collection of JSJ tori $\cT'$. Our main strategy is to choose flat metrics on the tori in $\cT'$ and then extend these metrics over each component of $M' \mid \cT'$ (Here, $M' \mid \cT'$ denote the manifold obtained by removing the interiors of tubular neighborhoods of the tori in $\cT'$) using a criterion by Leeb \cite{Leeb1995NPC3Manifolds}, which we restate in the language of charges of framed Seifert fibered spaces \cite{BuyaloSvetlov2005,LueckeWu1997}. We call this criterion the \emph{metric extension criterion}, abbreviated MEC.

To choose flat metrics on $\cT'$ satisfying MEC on each component of $M' \mid \cT'$, we lift the tori in $\cT'$ to a collection of incompressible tori $\cT$ in $M$ and construct a $G$-invariant flat metric $\bar g$ on the tori in $\cT$ satisfying MEC on each component of $M \mid \cT$. Since $\bar g$ is $G$-invariant, it descends to a flat metric $p_*\bar g$ on $\cT'$. We then verify that $p_*\bar g$ satisfies MEC on each component of $M' \mid \cT'$.

It remains to construct the metric $\bar g$ on the tori in $\cT$. We first choose an NPC metric $g$ on $M$ whose restriction to each torus in $\cT$ is flat. We then apply a modified averaging argument to the flat metric $g|_{\cT}$ with respect to the action of $G$. The resulting metric $\bar g$ is $G$-invariant and remains flat. Moreover, since MEC is a linear condition and $g|_{\cT}$ satisfies MEC on each component of $M \mid \cT$, the averaged metric $\bar g$ also satisfies MEC. Therefore $p_*\bar g$ satisfies MEC on $M'$, so the flat metrics on $\cT'$ extend over the JSJ pieces of $M'$ to give an NPC metric on $M'$.

\subsection{Future Directions}

One may ask whether our approach to Theorem \ref{thm:npc_metric_free_action_invariant} can be generalized to prove the analogous statement without the freeness assumption on the group action.

\medskip

\begin{conjecture} \label{conj:npc_dimension_3}
    Let $M$ be a closed orientable $3$-manifold admitting an NPC metric, and let $G$ be a finite group acting on $M$ by orientation-preserving diffeomorphisms. Then $M$ admits a $G$-invariant NPC metric.
\end{conjecture}

When the group $G$ in Conjecture \ref{conj:npc_dimension_3} does not act freely on $M$, the quotient $M' = M/G$ is an orientable $3$-orbifold rather than a $3$-manifold. There is a decomposition theorem for such orbifolds \cite[Theorem 3.15]{BoileauMaillotPorti2003} analogous to the JSJ decomposition of irreducible $3$-manifolds: an irreducible $3$-orbifold can be canonically split along a collection of incompressible, mutually disjoint, and mutually nonparallel Euclidean $2$-suborbifolds. The resulting components are either atoroidal or Seifert fibered $3$-orbifolds. This suggests a way to transform Conjecture \ref{conj:npc_dimension_3} into a topological problem about the orbifold JSJ decomposition. 

\subsection{Paper Organization}

This paper is organized as follows. 
In \S\ref{sec:npc}, we review some background concepts in NPC graph manifolds that are used in this paper. 
In \S\ref{sec:proof}, we prove Theorem \ref{thm:npc_metric_free_action_invariant} as outlined above. 

\subsection{Acknowledgements} \label{subsec:acknowledgements}

I would like to thank Bena Tshishiku for introducing me to Problem \ref{prob:npc_isometric}. This work would not have been possible without the many hours of fruitful and enthusiastic discussions we shared. I am especially grateful for his guidance on the proofs of Proposition \ref{prop:aspherical_implies_irreducible}, Lemma \ref{lem:faithful_torus_action_is_free_transitive}, and Lemma \ref{lem:properties_of_diff_of_torus}. I would also like to thank Ethan Dlugie and Eduardo Ventilari Sodr\'e for helpful discussions about the material in $\S$\ref{subsec:proof_of_main_2}.
\section{3-Manifolds with NPC Metrics} \label{sec:npc}

In this section, we review several properties of $3$-manifolds admitting NPC metrics that will be used in the proof of Theorem~\ref{thm:npc_metric_free_action_invariant}. We begin in $\S$\ref{subsec:JSJ_decomposition} by recalling the JSJ decomposition of irreducible $3$-manifolds. In particular, we show that every closed NPC $3$-manifold is irreducible and therefore admits a JSJ decomposition. Then, in $\S$\ref{subsec:framing_and_characteristic_covers}, we introduce unitary framings for Seifert fibered spaces with nonempty boundary. We also recall the notion of a characteristic covering of Seifert fibered spaces and show that unitary framings lift across such covers. Finally, in $\S$\ref{subsec:leeb's_condition}, we recall Leeb's metric extension criterion and provide a self-contained proof using the language of unitary framings and charges. 

We assume familiarity with the basic theory of surfaces and $2$-orbifolds. For a detailed introduction to $2$-orbifolds, we refer the reader to \cite{scott3manifolds,BoileauMaillotPorti2003}.

\subsection{JSJ Decomposition of an Irreducible 3-Manifold} \label{subsec:JSJ_decomposition}

Let $M$ be a closed orientable $3$-manifold. A smoothly embedded sphere $\phi: S^2 \hookrightarrow M$ is \emph{essential} if it does not bound a $3$-ball in $M$. A $3$-manifold $M$ is \emph{irreducible} if it does not contain any essential spheres.
An embedded torus $T^2 \subset M$ is \emph{incompressible} if the map $\pi_1(T^2) \hookrightarrow \pi_1(M)$ is injective. Geometrically, this means that any embedded disk $D^2$ in $M$ with $\partial D^2 \subset T^2$ bounds a disk in $T^2$. Given two embedded tori $T_1, T_2 \subset M$, we say that $T_1$ and $T_2$ are \emph{parallel} if there exists an embedded $T^2 \times I \subset M$ such that $T^2 \times \{0\} = T_1$ and $T^2 \times \{1\} = T_2$. We say that $T_1$ and $T_2$ are \emph{nonparallel} otherwise. Given a finite collection of mutually disjoint embedded tori $\cT = \{T_1, \ldots, T_n\}$ in $M$, let $N(T_i) \cong T_i \times I \subset M$ be a collection of mutually disjoint tubular neighborhoods of $T_i$ in $M$. We denote by $M \mid \cT$ the \emph{splitting of $M$ along $\cT$}, where
\begin{equation*}
    M \mid \cT = M \setminus \bigcup_{i=1}^n N(T_i).
\end{equation*}
If an orientable $3$-manifold is irreducible, then one can decompose $M$ along a finite collection of mutually disjoint, mutually nonparallel embedded tori $\cT$. The following theorem is due to Jaco--Shalen \cite{JacoShalen1979} and Johannson \cite{Johannson1979} and appears frequently in the literature as the \emph{JSJ decomposition}. See \cite[Theorem 1.9]{hatcher3manifolds} for a more modern proof.

\medskip

\begin{theorem}[JSJ Decomposition] \label{thm:JSJ_decomposition}
    Given a compact irreducible orientable $3$-manifold $M$, there exists a collection $\cT$ of finitely many mutually disjoint, mutually nonparallel incompressible tori such that each component of $M \mid \cT$ is either atoroidal or a Seifert fibered space. Such collection $\cT$ is unique up to isotopy.
\end{theorem}

We say that a collection of incompressible tori $\cT$ in $M$ is \emph{JSJ collection} if $\cT$ satisfies the conditions in Theorem \ref{thm:JSJ_decomposition}. We call the corresponding decomposition $M \mid \cT$ a \emph{JSJ decomposition} of $M$. By geometrization of 3-manifolds, each piece arising in the JSJ decomposition is geometric: if a component $X_i$ of $M \mid \cT$ is Seifert fibered, then it is modeled on one of the six Thurston geometries (cf.\ \cite{Thurston1997}); otherwise, $X_i$ is atoroidal, and the interior of $X_i$ is a finite-volume hyperbolic cusped manifold (cf.\ \cite{BessieresBessonBoileauMaillotPorti2010}). 

A crucial part of our proof of Theorem \ref{thm:npc_metric_free_action_invariant} is understanding how the incompressible tori in the JSJ decomposition of the quotient manifold $M/G$ relate to those in the JSJ decomposition of $M$. The following lemma shows that a JSJ collection of tori in $M/G$ always pulls back to a collection of incompressible tori in $M$. However, the pullback tori need not be mutually non-parallel. In this case, the corresponding splitting of $M$ may contain pieces diffeomorphic to $T^2 \times I$. This phenomenon is related to the presence of twisted $I$-bundles over the Klein bottle, denoted $K \tilde\times I$, in the JSJ decomposition of $M/G$.

\medskip

\begin{lemma} \label{lem:JSJ_of_quotient_pulls_back_to_JSJ_of_original}
    Let $M'$ be a closed orientable irreducible $3$-manifold that admits a finite-sheeted cover by a closed $3$-manifold $M$, with covering projection $p: M \rightarrow M'$. Let $\cT'$ be the collection of JSJ tori in $M'$ as in Theorem \ref{thm:JSJ_decomposition}. Let $\cT$ be the collection of connected components of $p^{-1}(T')$ over all $T' \in \cT'$. Then $\cT$ is a collection of incompressible tori in $M$ such that $M \mid \cT$ consists of Seifert fibered or atoroidal components. In particular, $\cT$ is a JSJ collection of $M$ if $M' \mid \cT'$ contains no components diffeomorphic to $K \tilde\times I$.
\end{lemma}

\begin{proof}
    We first note that every connected component of $p^{-1}(T')$ is a torus. Since $T'$ is incompressible in $M'$, each component of $p^{-1}(T')$ is incompressible in $M$.
    The components of $M \mid \cT$ are connected components of $p^{-1}(X')$, where $X'$ is a component of $M' \mid \cT'$. Since each $X'$ is either Seifert fibered or atoroidal, every finite cover of $X'$ is again Seifert fibered or atoroidal. Hence every component of $M \mid \cT$ is Seifert fibered or atoroidal.

    It remains to prove the mutually non-parallel statement. The only possible obstruction to $\cT$ being mutually non-parallel is that $M \mid \cT$ contains a product region $T^2 \times I$ between two parallel tori of $\cT$. Such a component covers a component of $M' \mid \cT'$ that is finitely covered by $T^2 \times I$. By the classification of orientable irreducible $3$-manifolds finitely covered by $T^2 \times I$, this component must be diffeomorphic to either $T^2 \times I$ or $K \tilde\times I$. 
    Since $\cT'$ is a JSJ collection of $M'$, no component of $M' \mid \cT'$ is diffeomorphic to $T^2 \times I$. By assumption, no component of $M' \mid \cT'$ is diffeomorphic to $K \tilde\times I$ either. Therefore $M \mid \cT$ contains no $T^2 \times I$ components, so no torus in $\cT$ is redundant. Hence $\cT$ is a JSJ collection of $M$.
\end{proof}

\medskip

\begin{example}
    Figure \ref{fig:twisted_klein_bottle} illustrates an example of a covering map $M \rightarrow M'$ where the JSJ tori of $M'$ lifts to parallel tori. On the left, the graph manifold $M$ is obtained by gluing two diffeomorphic copies $X_+$ and $X_-$ of a Seifert fibered space $X$ with two torus boundary components. The gluing is performed along an orientation-reversing involution $\phi: T^2 \rightarrow T^2$ such that $T^2/\langle \phi \rangle$ is the Klein bottle. The JSJ decomposition of $M$ consists of the two black tori $T_1$ and $T_2$. On the right, the quotient $M'$ is obtained from $M$ by a free involution that exchanges the pieces $X_+$ and $X_-$, sending $x \in \partial X_+$ to $\phi(x) \in \partial X_-$ and vice versa. The green tori $T'_1$ and $T'_2$ are the JSJ tori of $M'$, and each lifts to two tori $T_{i+}$ and $T_{i-}$ in $M$. Notice that each $T'_i$ bounds a copy of $K \tilde\times I$ in $M'$, which lifts to $T^2 \times I$ in $M$. Thus, for each $i=1,2$, the pair $T_{i+}$ and $T_{i-}$ consists of parallel tori isotopic to the JSJ torus $T_i$.
\end{example}

\begin{figure}[ht]
    \centering
    \footnotesize
    
    \def\svgwidth{0.9\columnwidth}
    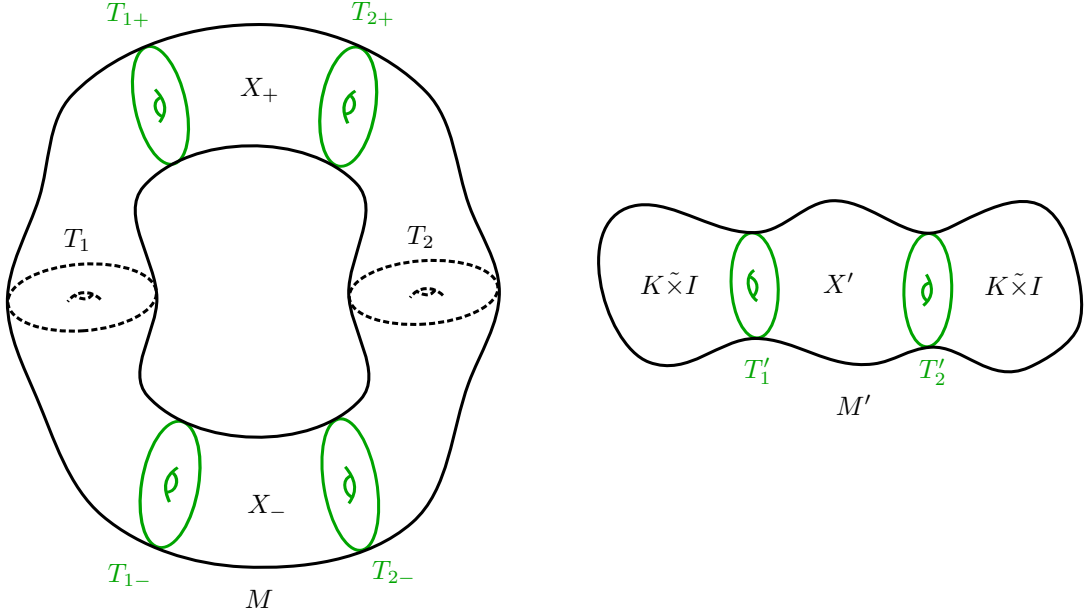

    \caption{An example in which the lifts of the JSJ tori of the quotient contains parallel tori.}
    \label{fig:twisted_klein_bottle}
\end{figure}

To conclude this section, we show that every NPC $3$-manifold is irreducible, so we can apply Theorem \ref{thm:JSJ_decomposition} on it. We first prove a general theorem about the relationship between manifolds $M$ whose universal cover is $\RR^n$ and embedded spheres in $M$.

\medskip

\begin{proposition} \label{prop:aspherical_implies_irreducible}
    Let $M$ be a connected, closed, orientable $n$-manifold whose universal cover is $\RR^n$, where $n \geq 3$. Then any smooth embedding of $S^{n-1}$ into $M$ bounds a topological $B^n$.
\end{proposition}

\begin{proof}
    We need to show that given any smooth embedding of an $(n-1)$-sphere $\phi: S^{n-1} \hookrightarrow M$, there exists an embedding of a closed $n$-ball $\psi: B^n \hookrightarrow M$ such that $\psi|_{\partial B^n} = \phi$. Let $\tilde p: \RR^n \rightarrow M$ be the universal covering projection. Since $S^{n-1}$ is simply connected for $n \geq 3$, the lifting property of covering maps implies that $\phi$ lifts to a smooth map $\tilde \phi: S^{n-1} \rightarrow \RR^n$. Because $\phi$ is an embedding and $\tilde p$ is a covering projection, $\tilde \phi$ is also a smooth embedding. By the generalized Schoenflies theorem, there exists a topological embedding $\tilde\psi: B^n \hookrightarrow \RR^n$ such that $\tilde\psi|_{\partial B^n} = \tilde\phi$.

    Consider the map $\psi = \tilde p \circ \tilde\psi$. By construction, $\psi|_{\partial B^n} = \phi$. To see that $\psi(B^n)$ is an embedded ball, it suffices to show that $\psi$ is injective. Assume for contradiction that $\psi$ is not injective. Then there exist two distinct points in $\tilde\psi(B^n)$ with the same image in $M$ under $\tilde p$, so there exists a nontrivial deck transformation $\gamma \in \Deck(\RR^n \xrightarrow{\tilde p} M)$ sending one to the other. In particular,
    \begin{equation*}
        \gamma(\tilde\psi(B^n)) \cap \tilde\psi(B^n) \neq \emptyset.
    \end{equation*}
    Since $\tilde\psi|_{\partial B^n} = \tilde\phi$, and since $\tilde\phi$ is a lift of the embedding $\phi$, we have
    \begin{equation*}
        \gamma(\tilde\phi(S^{n-1})) \cap \tilde\phi(S^{n-1}) = \emptyset.
    \end{equation*}
    Hence $\gamma(\tilde\phi(S^{n-1}))$ is disjoint from the boundary of $\tilde\psi(B^n)$. Since $\tilde\phi(S^{n-1})$ separates $\RR^n$ into two components and $\tilde\psi(B^n)$ is the closure of the bounded one, it follows, after replacing $\gamma$ by $\gamma^{-1}$ if necessary, that
    \begin{equation*}
        \gamma(\tilde\psi(B^n)) \subset \tilde\psi(B^n).
    \end{equation*}
    Since $\tilde\psi(B^n) \cong B^n$, the Brouwer fixed-point theorem implies that $\gamma$ has a fixed point in $\tilde\psi(B^n)$. This contradicts the fact that a nontrivial deck transformation acts freely.
\end{proof}

\medskip

\begin{corollary} \label{cor:aspherical_implies_irreducible}
    Every connected closed orientable $3$-manifold $M$ that admits an NPC metric is irreducible.
\end{corollary}

\begin{proof}
    The Cartan--Hadamard theorem implies that the universal cover of $M$ is $\RR^3$. The corollary then follows from Proposition \ref{prop:aspherical_implies_irreducible}.
\end{proof}

\subsection{Framings and Characteristic Coverings} \label{subsec:framing_and_characteristic_covers}

To study when a flat metric on the boundary of an orientable Seifert fibered space extends to an NPC metric on the whole space, we recall the notion of a framed Seifert fibered space.
Let $X \xrightarrow{\pi} O$ be an aspherical, orientable Seifert fibered space with nonempty boundary. Denote by $\partial X = T_1 \sqcup \cdots \sqcup T_n$ the torus components of the boundary of $X$. We orient the $T_j$'s by the orientations induced from the orientation of $X$. This gives intersection forms
\begin{equation*}
    \wedge_j: H_1(T_j; \ZZ) \times H_1(T_j; \ZZ) \rightarrow \ZZ.
\end{equation*}

We define a \emph{unitary framing} of $X$, which is a special type of framing as introduced in \cite{LueckeWu1997}. For each boundary torus $T_j$, let $[f_j]$ be the class in $H_1(T_j;  \ZZ)$ represented by an oriented regular fiber. (We will henceforth denote by $[c]$ the homology class of the (multi)curve $c$.) A \emph{framing} $B$ of $X$ is a union of oriented simple closed curves $B = b_1 \cup \ldots \cup b_n$, where $b_j \subset T_j$ and $[b_j] \wedge_j [f_j] \neq 0$. 
We call the pair $(X, B)$ a \emph{framed Seifert fibered space}. 
A framing is \emph{unitary} if we further require $[b_j] \wedge_j [f_j] = 1$. Equivalently, the ordered pair $([b_j], [f_j])$ is an oriented basis for $H_1(T_j; \ZZ)$. 

Let $i_j: T_j \hookrightarrow X$ be the inclusion map. The \emph{charge} of the framed Seifert fibered space $(X,B)$ is the rational number $\ch(X,B)$ satisfying
\begin{equation} \label{eqn:relative_euler_number}
    \sum_{j=1}^n (i_j)_* [b_j] = \ch(X,B) \, [f].
\end{equation}
Here, we work in $H_1(X;\QQ)$, and $[f]$ is the class of a regular fiber in $H_1(X;\QQ)$. See \cite[Lemma~1.1]{BuyaloSvetlov2005} for well-definedness of $\ch(X, B)$.

We now relate the charge $\ch(X, B)$ to the \emph{relative Euler number} $e(X,B)$ of a framed Seifert fibered space $(X,B)$ and to the Seifert invariants of the exceptional fibers of $X$. Suppose that $X$ has $m$ exceptional fibers with Seifert invariants $(\alpha_k,\beta_k)_{k=1}^m$. Let $c_k$ be a horizontal curve around the $k$-th exceptional fiber such that $\alpha_k [c_k] + \beta_k [f] = 0$ in $H_1(X; \ZZ)$. The existence of $c_k$ follows from the interpretation of Seifert invariants in \cite{LueckeWu1997}. There exists $e_0 \in \ZZ$ that satisfies
\begin{equation*}
    \sum_{j=1}^n (i_j)_* [b_j] + \sum_{k=1}^m [c_k] = e_0 [f].
\end{equation*}
Working in $H_1(X; \QQ)$ and replacing $[c_k]$ by $-\frac{\beta_k}{\alpha_k} [f]$, we obtain
\begin{equation*}
    \ch(X,B) [f] 
    = \sum_{j=1}^n (i_j)_* [b_j] 
    = \left(e_0 + \sum_{k=1}^m \frac{\beta_k}{\alpha_k}\right) [f].
\end{equation*}
The coefficient of $[f]$ on the right-hand side is precisely $-e(X(B))$, where $X(B)$ is the closed Seifert fibered space obtained from $X$ by Dehn filling each boundary torus $T_j$ so that $b_j$ is the meridian of the corresponding filling solid torus. The quantity $e(X,B) = e(X(B))$ is called the \emph{relative Euler number} of the framed Seifert fibered space $(X,B)$; cf.\ \cite{LueckeWu1997}. Thus, with these conventions,
\begin{equation} \label{eqn:charge_is_negative_rel_euler}
    \ch(X,B) = -e(X,B).
\end{equation}
We remark that Equation \eqref{eqn:charge_is_negative_rel_euler} was observed by Buyalo--Svetlov in the discussion following \cite[Lemma 1.1]{BuyaloSvetlov2005}.

We would like to study how unitary framings behave under finite covers of Seifert fibered spaces. We recall from \cite{LueckeWu1997} that, given framed Seifert fibered spaces $(X,B)$ and $(X',B')$, a map $p: X \rightarrow X'$ is called a \emph{covering of framed Seifert fibered spaces} if $p$ is fiber-preserving and each component of $B$ is a component of $p^{-1}(B')$. We often abbreviate this by saying that $p: (X,B) \rightarrow (X',B')$ is a covering.

Luecke--Wu studied the naturality of $e(X, B)$ with respect to finite covers of framed Seifert fibered spaces. A covering projection $p: X \rightarrow X'$ of Seifert fibered spaces is a \emph{$(u,v)$-cover} if each regular fiber of $X'$ is covered by $u$ regular fibers of $X$, and the restriction of $p$ to each regular fiber has degree $v$. Luecke--Wu showed that if $p: (X, B) \rightarrow (X',B')$ is a $(u,v)$-cover of framed Seifert fibered spaces, then
\begin{equation*}
    e(X,B) = \frac{u}{v} \, e(X',B').
\end{equation*}
Equation \eqref{eqn:charge_is_negative_rel_euler} then gives us
\begin{equation} \label{eqn:naturality_of_charge}
    \ch(X,B) = \frac{u}{v} \, \ch(X',B').
\end{equation}
This multiplicative property of $\ch(X, B)$ is convenient when we consider coverings of the form $X \rightarrow X/G$, where $G$ acts on $X$ freely by orientation-preserving, fiber-preserving diffeomorphisms. However, there is one issue: given a fiber-preserving cover $p: X \rightarrow X'$, there need not exist unitary framings $B$ and $B'$ on $X$ and $X'$, respectively, such that $p: (X,B) \rightarrow (X',B')$ is a covering of framed Seifert fibered spaces. The obstruction comes from possible shearing in the induced coverings of the boundary tori.

The issue described above does not arise for characteristic coverings. Let $p: T \rightarrow T'$ be a covering of tori. After choosing basepoints, let $\pi' = \pi_1(T')$ and let $\pi < \pi'$ be the finite-index subgroup corresponding to $\pi_1(T)$. We recall from \cite{LueckeWu1997} that $p: T \rightarrow T'$ is an \emph{$s$-characteristic covering} if $\pi = s\pi'$, where $s\pi'$ denotes the unique $s$-characteristic subgroup of $\pi' \cong \ZZ^2$. Equivalently, after identifying $\pi'$ with $\ZZ^2$, the subgroup $\pi$ is $(s\ZZ)^2$. A fiber-preserving covering of Seifert fibered spaces $p: X \rightarrow X'$ is \emph{$s$-characteristic} if its restriction to each boundary torus of $X$ is an $s$-characteristic covering. 

We show that unitary framings can always be lifted across characteristic covers.

\medskip
\begin{proposition} \label{prop:charge_and_characteristic_covering}
    Let $p: X \rightarrow X'$ be an $s$-characteristic degree-$d$ fiber-preserving covering of oriented Seifert fibered spaces. Let $B'$ be a unitary framing of $X'$. Then the following statements hold.
    \begin{enumerate}
        \item[\textnormal{(a)}] There exists a unitary framing $B$ of $X$ such that $p: (X,B) \rightarrow (X',B')$ is a covering, and
        \begin{equation} \label{eqn:charge_and_characteristic_covering}
        \ch(X,B) = \frac{d}{s^2} , \ch(X',B').
        \end{equation}

        \item[\textnormal{(b)}] If $p: X \rightarrow X'$ is regular, and $G = \Deck(X \rightarrow X')$ is the finite deck group such that $X' \cong X/G$, then $B$ can be chosen so that $\phi_*[B] = [B]$ for every $\phi \in G$, and $p_*[B] = \frac{d}{s}[B']$.
        Here $[B]$ and $[B']$ denote the sums of the homology classes of the framing curves in $H_1(X;\ZZ)$ and $H_1(X';\ZZ)$, respectively. Moreover, componentwise, if $p(T_j)=T'_k$, then
        \begin{equation*}
            p_*[b_j] = s[b'_k].
        \end{equation*}
    \end{enumerate}
\end{proposition}

\begin{proof}
    We first prove \textnormal{(a)}. We construct the unitary framing $B$ as follows. Let $T_j \subset \partial X$ be a boundary torus, and let $T'_k \subset \partial X'$ be the boundary torus such that $p(T_j)=T'_k$. Write $b'_k = B' \cap T'_k$, and let $f'_k$ be the class of an oriented regular fiber on $T'_k$. Since $B'$ is unitary, the classes $[b'_k]$ and $[f'_k]$ form an oriented basis for the lattice $\Gamma'_k \cong H_1(T'_k;\ZZ)$ and satisfy
    \begin{equation*}
    [b'_k] \wedge'_k [f'_k] = 1.
    \end{equation*}

    Since the restriction $p|_{T_j}: T_j \rightarrow T'_k$ is an $s$-characteristic cover, the corresponding sublattice $\Gamma_j < \Gamma'_k$ is
    \begin{equation*}
        \Gamma_j
        =
        s\ZZ \langle [b'_k] \rangle
        \oplus
        s\ZZ \langle [f'_k] \rangle .
    \end{equation*}
    Since $p$ is fiber-preserving, the class $[f_j]$ of an oriented regular fiber of $X$ on $T_j$ satisfies $p_*[f_j] = s[f'_k]$. 
    Choose $b_j$ to be a component of $p^{-1}(b'_k) \cap T_j$, oriented so that $p_*[b_j] = s[b'_k]$ in the sublattice $\Gamma_j < \Gamma'_k$. Then $b_j$ is primitive in $\Gamma_j$, and with respect to the induced orientation on $T_j$ we have
    \begin{equation*}
        [b_j] \wedge_j [f_j] = 1.
    \end{equation*}
    Doing this for every boundary component of $X$ defines a unitary framing $B$. Moreover, each $b_j$ is a component of $p^{-1}(B')$, so $p: (X,B) \rightarrow (X',B')$ is a covering of framed Seifert fibered spaces.

    It remains to compute the scaling factor. Since $p$ is $s$-characteristic on each boundary torus, the restriction of $p$ to each regular fiber has degree $s$. Hence $p$ is a $(d/s,s)$-cover in the sense of Luecke--Wu. Equation~\eqref{eqn:charge_and_characteristic_covering} then follows from Equation~\eqref{eqn:naturality_of_charge}.

    We now prove \textnormal{(b)}. Assume that $p$ is regular. Let $\partial X' = T'_1 \sqcup \cdots \sqcup T'_{n'}$.
    For each $k$, write $b'_k = B' \cap T'_k$. Since each boundary component of $X$ covers its image in $\partial X'$ with degree $s^2$, the preimage $p^{-1}(T'_k)$ consists of $l = \frac{d}{s^2}$ 
    many tori. Write $p^{-1}(T'_k) = T_{k,1} \sqcup \cdots \sqcup T_{k,l}$.

    Fix $k$. Since $p: X \rightarrow X'$ is regular, the deck group $G$ acts transitively on the collection of tori $\{T_{k,1},\ldots,T_{k,l}\}$. Let $S_k < G$ be the stabilizer of $T_{k,1}$. Then $S_k$ is the deck group of the restricted cover $p|_{T_{k,1}}: T_{k,1} \rightarrow T'_k$.
    Since this restricted cover is $s$-characteristic, we have $S_k \cong (\ZZ/s\ZZ)^2$.
    In particular, $S_k$ acts by translations on $T_{k,1}$, and therefore every element of $S_k$ acts trivially on $H_1(T_{k,1};\ZZ)$.

    Choose a curve $b_{k,1} \subset T_{k,1}$ as in the proof of \textnormal{(a)}. Thus $p_*[b_{k,1}] = s[b'_k]$. For each $j=1,\ldots,l$, choose an element $\phi_{k,j} \in G$ such that $\phi_{k,j}(T_{k,1}) = T_{k,j}$, and notice that $\phi_{k,j}(b_{k,1})$ is parallel to $b_{k,j} = B \cap T_{k,j}$, so $(\phi_{k,j})_*[b_{k,1}] = [b_{k,j}]$ and 
    \begin{equation*}
        p_*[b_{k,j}] = p_* (\phi_{k,j})_* [b_{k,1}] = s[b'_k].
    \end{equation*}

    We first compute the pushforward of $[B]$ by $p$. For each $k$, define $[B_k] = \sum_{j=1}^l [b_{k,j}]$.
    Then
    \begin{equation*}
        [B] 
        = \sum_{k=1}^{n'} [B_k] 
        = \sum_{k=1}^{n'} \sum_{j=1}^{l} [b_{k,j}].
    \end{equation*}
    Since $p_*[b_{k,j}] = s[b'_k]$ for every $j$, we get
    \begin{equation*}
        p_*[B]
        = \sum_{k=1}^{n'} \sum_{j=1}^l p_*[b_{k,j}]
        = \sum_{k=1}^{n'} \sum_{j=1}^l s[b'_k]
        = ls \sum_{k=1}^{n'} [b'_k]
        = \frac{d}{s}[B'].
    \end{equation*}

    It remains to show that $\phi_*[B]=[B]$ for every $\phi \in G$. Since each deck transformation preserves $p^{-1}(T'_k)$ setwise, it suffices to prove $\phi_*[B_k] = [B_k]$ for every $k$.

    Fix $k$ and $\phi \in G$. The deck transformation $\phi$ permutes the tori $T_{k,1},\ldots,T_{k,l}$. Let $\sigma \in \cS_l$ be the corresponding permutation such that $\phi(T_{k,j}) = T_{k,\sigma(j)}$. 
    For each $j$, the element $\phi\phi_{k,j}$ sends $T_{k,1}$ to $T_{k,\sigma(j)}$. Hence there exists $\psi_j \in S_k$ such that $\phi\phi_{k,j} = \phi_{k,\sigma(j)}\psi_j$.
    Since every element of $S_k$ acts trivially on $H_1(T_{k,1};\ZZ)$, we have $(\psi_j)_*[b_{k,1}] = [b_{k,1}]$. 
    Therefore
    \begin{equation*}
        \phi_*[B_k]
        = \sum_{j=1}^l \phi_*[b_{k,j}]
        = \sum_{j=1}^l (\phi\phi_{k,j})_*[b_{k,1}]
        = \sum_{j=1}^l (\phi_{k,\sigma(j)})_*(\psi_j)_*[b_{k,1}]
        = \sum_{j=1}^l [b_{k,\sigma(j)}]
        = [B_k].
    \end{equation*}
    This proves $\phi_*[B]=[B]$ for every $\phi \in G$, and completes the proof.
\end{proof}

\medskip

\begin{remark} \label{rmk:generality_of_characteristic_coverings}
    Characteristic covers are sufficiently flexible for our purposes. Indeed, any finite cover $M \rightarrow M'$ of graph manifolds can be refined to a finite cover
    \begin{equation*}
    \widetilde M \longrightarrow M \longrightarrow M'
    \end{equation*}
    such that the composite $\widetilde M \rightarrow M'$ is characteristic. The existence of such a characteristic refinement is proved in \cite[Proposition~4.4]{LueckeWu1997}. Moreover, after replacing $\widetilde M$ by the finite cover corresponding to the normal core of $\pi_1(\widetilde M)$ in $\pi_1(M')$, we may assume that the composite cover $\widetilde M \rightarrow M'$ is both regular and characteristic; see \cite[Proposition~4.2]{DerbezLiuWang2015}.

    Let $\widetilde G$ be the deck group of the regular cover $\widetilde M \rightarrow M'$. Then $M$ admits a $G$-invariant NPC metric, where $G=\Deck(M \rightarrow M')$, if and only if $\widetilde M$ admits a $\widetilde G$-invariant NPC metric. Indeed, a $G$-invariant NPC metric on $M$ descends to an NPC metric on $M'$, whose pullback to $\widetilde M$ is $\widetilde G$-invariant. Conversely, a $\widetilde G$-invariant NPC metric on $\widetilde M$ descends to an NPC metric on $M'$, whose pullback to $M$ is $G$-invariant. Therefore, in the proof of Theorem~\ref{thm:npc_metric_free_action_invariant}, we may assume that the covering $M \rightarrow M/G$ is characteristic.
\end{remark}

\subsection{The Metric Extension Criterion (MEC)} \label{subsec:leeb's_condition}

Here we recall the metric extension criterion stated by Leeb \cite[Lemma 2.5]{Leeb1995NPC3Manifolds}. We use the same convention for the boundary components of a compact Seifert fibered space $X \xrightarrow{\pi} O$ as in $\S$\ref{subsec:framing_and_characteristic_covers}. Since $X$ is aspherical, its regular fiber determines an infinite cyclic normal subgroup of $\pi_1(X)$, and $\pi_1(X)$ fits into the short exact sequence
\begin{equation} \label{eqn:ses_of_seifert_fibered_spaces_fundamental_group}
    1 \longrightarrow \langle f \rangle \cong \mathbb Z
    \overset{i_*}{\longrightarrow} \pi_1(X)
    \overset{\pi_*}{\longrightarrow} \pi_1^{\orb}(O)
    \longrightarrow 1,
\end{equation}
where $i: f \hookrightarrow X$ is the inclusion of a regular fiber. In particular, if the base orbifold $O$ is orientable, then $\langle f\rangle$ lies in the center of $\pi_1(X)$. 

A Riemannian metric $g$ on $X$ is \emph{compatible with the Seifert fibration of $X$} if all regular fibers of $X$ are geodesic curves of the same length.
If $g$ is NPC, then the universal Riemannian cover decomposes as a Riemannian product
\begin{equation*}
    (\tilde X, \tilde g) \cong Y \times \RR,
\end{equation*}
where $Y$ is non-positively curved, and $\RR$ is equipped with the standard Euclidean metric, cf.\ \cite[Lemma 1]{Eberlein1982}. The $\RR$-direction projects to the fibers of $X$ in the sense that the following diagram commutes:
% https://q.uiver.app/#q=WzAsNixbMSwwLCJZIFxcdGltZXMgXFxSUiJdLFsyLDAsIlkiXSxbMSwxLCJYIl0sWzIsMSwiTyJdLFswLDEsIlNeMSJdLFswLDAsIlxcUlIiXSxbNSwwLCJcXHRpbGRlIGkiXSxbMCwxLCJcXHRpbGRlXFxwaSJdLFsxLDNdLFs1LDRdLFs0LDIsImkiXSxbMCwyXSxbMiwzLCJcXHBpIl1d
\[\begin{tikzcd}
	\RR & {Y \times \RR} & Y \\
	{S^1} & X & O
	\arrow["{\tilde i}", from=1-1, to=1-2]
	\arrow[from=1-1, to=2-1]
	\arrow["{\tilde\pi}", from=1-2, to=1-3]
	\arrow[from=1-2, to=2-2]
	\arrow[from=1-3, to=2-3]
	\arrow["i", from=2-1, to=2-2]
	\arrow["\pi", from=2-2, to=2-3]
\end{tikzcd}\]
Here, $i$ (resp.\ $\tilde i$) is the inclusion of a fixed $S^1$ (resp.\ $\RR$) fiber into $X$ (resp.\ $Y \times \RR$), and $\pi$ (resp.\ $\tilde\pi$) is the projection to the base orbifold (resp.\ universal cover of the base orbifold).
The representation $\rho: \pi_1(X) \rightarrow \Isom(Y \times \RR)$ associated with the universal Riemannian cover decomposes as a pair of representations $(\rho', \phi)$. Here,
\begin{equation} \label{eqn:fiber_deck_action}
    \phi: \pi_1(X) \rightarrow \Isom(\RR) \cong \RR \rtimes \ZZ/2\ZZ
\end{equation}
is the deck action of $\pi_1(X)$ in the $\RR$-direction. Notice that $\phi(f) \in \RR \setminus \{0\}$, where $f \in \pi_1(X)$ is the regular fiber class. On the other hand, $\rho': \pi_1^{\orb}(O) \rightarrow \Isom(Y)$ is a cocompact discrete action that induces an NPC orbifold metric on $O$. 

To prove MEC, we need to work with the following standard presentations for $\pi_1(X)$. Let $(\alpha_k, \beta_k)_{k=1}^m$ be the Seifert invariants of $X$. Then $O$ has $m$ cone points, say of orders $\alpha_1,\ldots,\alpha_m$. Let $g$ be the genus of $|O|$ after capping the boundary components of $|O|$. Assume first that $O$ is orientable. Then, given any unitary framing $B = b_1 \cup \ldots \cup b_n$ of $X$, using the short exact sequence from Equation \eqref{eqn:ses_of_seifert_fibered_spaces_fundamental_group}, we have the following presentation.
\begin{equation} \label{eqn:orientable_seifert_fundamental_group_pres}
    \pi_1(X) =
    \Big\langle
        x_i, y_i, b_j, c_k, f
        \,\Big|\,
        \text{$f$ central},\;
        \prod_{i=1}^g [x_i,y_i] \prod_{j=1}^n b_j \prod_{k=1}^m c_k f^{-e_0},\;
        c_k^{\alpha_k} f^{\beta_k}
    \Big\rangle.
\end{equation}
Here, the generators $x_i, y_i, c_k$ are lifts of curve classes in $\pi_1^{\orb}(O)$, where $\{x_i, y_i\}_{i=1}^g$ are curve classes that generate the fundamental group of a closed orientable genus-$g$ surface, and $\{c_k\}_{k=1}^k$ are curve classes of the curves around the cone points. The generators $\{b_j\}_{j=1}^n$ are curve classes of the unitary framing after changing basepoints. The generator $f$ is the regular fiber class as in Equation \eqref{eqn:ses_of_seifert_fibered_spaces_fundamental_group}. We call the last relation in Equation \eqref{eqn:orientable_seifert_fundamental_group_pres} involving the product of all generators the \emph{long relation}.

If $O$ is non-orientable, then we may simply replace the generators $\{x_i, y_i\}_{i=1}^g$ with the ``crosscap'' generators $\{z_i\}_{i=1}^g$ that generate the fundamental group a closed non-orientable genus-$g$ surface. This gives us the following presentation.
\begin{equation} \label{eqn:non-orientable_seifert_fundamental_group_pres}
    \pi_1(X) =
    \Big\langle
        z_i, b_j, c_k, f
        \,\Big|\,
        z_i f z_i^{-1}f,\;
        [f,b_j],\; [f,c_k],\;
        \prod_{i=1}^g z_i^2 \prod_{j=1}^n b_j \prod_{k=1}^m c_k f^{-e_0},\;
        c_k^{\alpha_k} f^{\beta_k}
    \Big\rangle.
\end{equation}
The presentations in Equations \eqref{eqn:orientable_seifert_fundamental_group_pres} and \eqref{eqn:non-orientable_seifert_fundamental_group_pres} can be derived from the Seifert--van Kampen theorem by decomposing the base orbifold into a punctured surface together with neighborhoods of the cone points, and then lifting this decomposition to the Seifert fibration. We omit the details and refer the readers to \cite[$\S$2.2]{BoileauMaillotPorti2003} and \cite{Orlik1972,scott3manifolds}.

If $X$ admits an NPC metric with totally geodesic boundary, then the Cartan--Hadamard theorem implies that $Y$ is a simply-connected subspace of $\RR^2$, so the base orbifold $O$ admits a Euclidean or hyperbolic metric with totally geodesic boundary. In particular, $X$ admits a locally homogeneous metric modeled on $\EE^3$ or $\HH^2 \times \RR$. 
However, if we are interested in NPC metrics on $X$ that restrict to specific metrics on the boundary components, which is always the case when we study manifolds by gluing Seifert fibered pieces together, then not every such $X$ admits an NPC metric extending a given collection of flat metrics on the boundary tori, cf.\ \cite{Leeb1995NPC3Manifolds,BuyaloSvetlov2005}. 

Leeb \cite[Lemma 2.5]{Leeb1995NPC3Manifolds} gave a homological criterion for when a flat metric on $\partial X$ can be extended to an NPC metric on $X$. We call this criterion MEC, and we recall and state Leeb's result using the language of unitary framings and charge. 

A flat metric $g$ on a torus $T^2$ induces an inner product $\sigma_g$ on the real vector space $H_1(T^2; \RR)$ by setting $\sigma_g([a], [a])$ to be the squared length of the curve $a$ for any primitive curve classes $[a] \in H_1(T^2; \ZZ)$. 
Given any diffeomorphism $f: T \rightarrow T$, we have
\begin{equation*}
    \sigma_{f^*g} ([a], [b]) = \sigma_g(f_*[a], f_*[b]).
\end{equation*}
In particular, if $f \in \Diff_0(T)$, then the induced maps $f_*$ and $f^*$ on homology are trivial, so $\sigma_{f^*g} = \sigma_g$.

\medskip

\begin{proposition}[MEC] \label{prop:leeb_condition}
    Given a compact orientable Seifert fibered space $X \rightarrow O$ with boundary components $T_1, \ldots, T_n$ such that $X$ admits a totally geodesic NPC metric compatible with its Seifert fibration, let $B = b_1 \cup \ldots \cup b_n$ be any unitary framing of $X$. 
    For each $j = 1, \ldots, n$, let $f_j$ be a regular fiber on $T_j$, and let $g_j$ be a flat metric on $T_j$. Then the following statements hold.
    \begin{enumerate}
        \item[\textnormal{(a)}] The metrics $g_j$ extend to an NPC metric $g$ on $X$ compatible with the Seifert fibration of $X$, for which all the $T_j$ are totally geodesic, if and only if there exists $l \in \RR_{>0}$ such that $\sigma_{g_j}([f_j], [f_j]) = l^2$ for all $j$, and
        \begin{equation} \label{eqn:leeb_condition}
            \sum_{j=1}^n \sigma_{g_j}([f_j], [b_j]) = \ch(X, B) \, l^2.
        \end{equation}
        Here, $l$ is the length of the regular fiber of $(X, g)$.

        \item[\textnormal{(b)}] The extended NPC metric $g$ can be constructed to be flat on a collar neighborhood of the boundary components.
    \end{enumerate}
\end{proposition}

\begin{proof}
    We first prove (a).
    Assume first that $X$ has an NPC metric $g$ restricting to each $g_j$ on $T_j$. Let $\rho: \pi_1(X) \rightarrow \Isom(Y \times \RR)$ be the isometric deck action of $\pi_1(X)$ on the universal Riemannian cover $Y \times \RR$. Consider the representation $\phi: \pi_1(X) \rightarrow \Isom(\RR)$ induced by $\rho$. Observe that $\phi(f)$ is the length of the regular fiber, where $f$ is the regular fiber generator in Equation \eqref{eqn:orientable_seifert_fundamental_group_pres} or \eqref{eqn:non-orientable_seifert_fundamental_group_pres}. Moreover, for any $x \in \pi_1(X)$ whose monodromy action on the fiber circle preserves the orientation of a regular fiber, $\phi(x)$ measures the signed translation distance of $\rho(x) \in \Isom(Y \times \RR)$ in the $\RR$-direction. Otherwise, $\phi(x)$ is a reflection. Since the monodromy action of $b_j$ preserves the orientation of a regular fiber, we have
    \begin{equation} \label{eqn:leeb_condition_proof_1}
        \sigma_{g_j}([f_j], \, p[f_j] + q[b_j]) = p \phi(f)^2 + q \phi(f)\phi(b_j).
    \end{equation}
    Let $l = \phi(f)$. 
    Taking $p = 1$ and $q = 0$, we see that $\sigma_{g_j}([f_j], [f_j]) = \phi(f)^2 = l^2$. On the other hand, taking $p = 0$ and $q = 1$, we get 
    \begin{equation*}
        \sum_{j=1}^n \sigma_{g_j}([f_j], [b_j])
        = \phi(f) \cdot \sum_{j=1}^n \phi(b_j)
    \end{equation*}
    Now, if the base orbifold $O$ is orientable, then all generators of $\pi_1(X)$ in Equation \eqref{eqn:orientable_seifert_fundamental_group_pres} preserve the orientation of a regular fiber of $X$, so the long relation in Equation \eqref{eqn:orientable_seifert_fundamental_group_pres} gives us 
    \begin{equation*}
        \sum_{j=1}^n \phi(b_j) + \sum_{k=1}^m \phi(c_k) - e_0 \phi(f) 
        = \sum_{j=1}^n \phi(b_j) + \left( - \sum_{k=1}^m \frac{\beta_k}{\alpha_k} - e_0 \right) \phi(f)
        = 0.
    \end{equation*}
    In other words, $\sum_{j=1}^n \phi(b_j) = \ch(X, B) \phi(f)$. Combining this with Equation \eqref{eqn:leeb_condition_proof_1} gives us Equation \eqref{eqn:leeb_condition}. The case where $O$ is non-orientable is similar except that the monodromy action of the ``crosscap'' generators $z_i$ in Equation \eqref{eqn:non-orientable_seifert_fundamental_group_pres} revert the orientation of a regular fiber, so $z_i^2$ is a reflection, which implies $\phi(z_i^2) = 0$.

    For the converse direction, suppose $B = b_1 \cup \ldots \cup b_n$ is a unitary framing of $X$ and that $\sigma_{g_j}([f_j], [f_j]) = l$ is independent of $j$. We assume that $X$ is modeled on $\HH^2 \times \RR$ (the Euclidean case is analogous, so we omit it). Choose $\rho': \pi_1^{\orb}(O) \rightarrow \Isom(\HH^2)$ so that the squared length $s_j^2$ of the $j$-th boundary curve is the length of the component of $b_j$ orthogonal to $f_j$ in the flat torus metric $g_j$:
    \begin{equation*}
        s_j^2 = \sigma_{g_j}([b_j], [b_j]) - \frac{\sigma_{g_j}([f_j], [b_j])^2}{\sigma_{g_j}([f_j], [f_j])}.
    \end{equation*}
    Such a representation exists because hyperbolic structures on $O$ can realize arbitrary boundary curve lengths. Construct $\phi: \pi_1(X) \rightarrow \Isom(\RR)$ by specifying the images of the generators in Equation \eqref{eqn:orientable_seifert_fundamental_group_pres} or \eqref{eqn:non-orientable_seifert_fundamental_group_pres}. Specifically, we let $\phi$ map $f$ to $\sqrt{\sigma_{g_j}([f_j], [f_j])}$ and $b_j$ to $\sigma_{g_j}([f_j], [b_j]) / \phi(f)$. We then map $c_k$ to $-\frac{\beta_k}{\alpha_k}\,\phi(f)$. If $O$ is orientable, then we send $x_i, y_i \mapsto 0$. If $O$ is non-orientable, we send $z_i$ to the reflection across $0$. To check that $\phi$ is well-defined, we must check that $\phi$ respects the relations in Equation \eqref{eqn:orientable_seifert_fundamental_group_pres} or \eqref{eqn:non-orientable_seifert_fundamental_group_pres}. The only non-trivial relation to check is the long relation, which holds by construction of $\phi$ and Equation \eqref{eqn:leeb_condition}.
    
    The pair $(\rho', \phi)$ then induces a cocompact discrete faithful representation $\rho: \pi_1(X) \rightarrow \Isom(\HH^2 \times \RR)$ by
    \begin{equation*}
        \rho(a) 
        = (\rho' \circ \pi_* (a), \phi(a)) 
        \in \Isom(\HH^2) \times \Isom(\RR) 
        \cong \Isom(\HH^2 \times \RR),
    \end{equation*}
    where $\pi_*: \pi_1(X) \rightarrow \pi_1^{\orb}(O)$ is the induced homomorphism. The metric $g$ induced by $\rho$ satisfies the desired properties.

    For (b), notice that if $X$ is modeled on $\EE^3$, then the geometric metric induced by the representation $\rho: \pi_1(X) \rightarrow \Isom(\EE^3)$ is already flat, so we only need to consider the case where $X$ is modeled on $\HH^2 \times \RR$. Since we are working locally near $\partial X$, we may remove the exceptional fibers of $X$ and treat it as a circle bundle over a hyperbolic surface $\Sigma$ with geodesic boundary curves. Near each boundary component $T_j \subset \partial X$, there is a unique geodesic $\gamma_j$ in $\HH^2$ that projects to $T_j \cap \Sigma$. We can choose $r_j > 0$ small enough so that the $r_j$-neighborhood $N_{r_j}$ of $T_j$ in $X$ pulls back under the universal covering map to
    \begin{equation*}
        \tilde N_{r_j} = \gamma_j \times \RR \times [0,r_j).
    \end{equation*}
    Here, $\tilde N_{r_j}$ projects to $\gamma_j \times [0,r_j) \subset \HH^2$. The metric on $\HH^2 \times \RR$ restricts to a warped product metric $\tilde g_j$ on $\tilde N_{r_j}$
    \begin{equation*}
        \tilde g_j = \cosh^2(t) \, ds^2 + d\theta^2 + dt^2.
    \end{equation*}
    Here, $s$ is the coordinate in the direction of the geodesic $\gamma_j$, $t$ is the coordinate in $\HH^2$ normal to the geodesic $\gamma_j$, and $\theta$ is the coordinate in the $\RR$-factor of $\HH^2 \times \RR$.
    In particular, $N_{r_j}$ is the quotient of $\tilde N_{r_j}$ by an isometric action of $\pi_1(T_j)$ given by a representation $\rho_j: \pi_1(T_j) \hookrightarrow \Isom(\tilde N_{r_j}, \tilde g_j)$ sending
    \begin{equation*}
        \rho_j(\alpha)(s, \theta, t) = (\rho_j'(\alpha)(s, \theta), t)
        \quad \text{for all } \alpha \in \pi_1(T_j).
    \end{equation*}
    Here, $\rho_j'$ is the action of $\pi_1(T_j)$ on $\pi^{-1}(T_j) = \gamma_j \times \RR \times \{0\}$.

    To flatten the metric near $T_j$, we take a smooth convex functon $f_j: [0,r_j) \rightarrow \RR_{>0}$ such that $f_j \equiv 1$ in a neighborhood of $0$, and $f_j(t) = \cosh(t)$ in a neighborhood of $r_j$. We replace $\tilde g_j$ with the following warped product metric:
    \begin{equation*}
        \tilde h_j := f_j^2(t)\, ds^2 + d\theta^2 + dt^2.
    \end{equation*}
    Observe that the metric $\tilde h_j$ is NPC (see \cite[$\S$7]{bishop-oneill-1969} for computing the sectional curvature of $\tilde h_j$). Also, using the original action $\rho_j$, $\pi_1(T_j)$ still acts by isometries on $(\tilde N_{r_j}, \tilde h_j)$. The quotient manifold $N_{r_j}$ is equipped with the metric $h_j$ that restricts to $g_j$ on $T_j$, is flat near $T_j$, and agrees with $g$ outside the $r_j-\varepsilon$-neighborhood of $T_j$ for some small $\varepsilon > 0$. This proves (b).
\end{proof}

\medskip

\begin{corollary} \label{cor:leeb_condition}
    Let $M$ be an orientable graph manifold with a collection of incompressible tori $\cT = \{T_1, \ldots, T_n\}$ such that the components
    \begin{equation*}
        M \mid \cT = X_1 \sqcup \cdots \sqcup X_m.
    \end{equation*}
    in the splitting are all Seifert fibered. Then, a collection of flat metrics $g_j$ on $T_j$ extends to an NPC metric $g$ on $M$ such that each $T_j$ is totally geodesic if, for each Seifert fibered component $X_i$, there exists a unitary framing $B_i$ satisfying the conditions in Proposition \ref{prop:leeb_condition}.
\end{corollary}

\begin{proof}
    By Proposition \ref{prop:leeb_condition}(a), we have a metric $h_i$ on each Seifert fibered component $X_i$ that restricts to $g_j$ on $T_j$ whenever $T_j \subset X_i$. Additionally, Proposition \ref{prop:leeb_condition}(b) implies we can make $h_i$ flat in a collar neighborhood of $\partial X_i$. We thus obtain a smooth NPC metric $g$ on $M$ by the following construction: Set $g|_{X_i} = h_i$, and set $g$ on tubular neighborhoods $T_j \times I$ in $M \setminus (M \mid \cT)$ to be the product metric $g_j + dt^2$. Since $h_i$ is flat in a collar neighborhood of $T_j \subset \partial X_i$ and $h_i|_{T_j} = g_j$, the metric $g$ is smooth. 
\end{proof}

\medskip

\begin{example}
    As mentioned earlier in the section, a graph manifold need not admit an NPC metric even when each Seifert fibered piece in its JSJ decomposition admits one. Figure \ref{fig:npc_examples} gives some simple examples of graph manifolds. Here we may assume that each JSJ piece $X_i$ is an $S^1$-bundle over a compact surface $\Sigma_i$. Then, $B_i = \Sigma_i \cap \partial X_i$ is a unitary framing such that $\ch(X_i, B_i) = 0$. By Proposition \ref{prop:leeb_condition} and Corollary \ref{cor:leeb_condition}, $M_1$ admits an NPC metric if and only if the induced map $\phi_*: H_1(T_-) \rightarrow H_1(T_+)$ sends $([f_-], [b_-])$ to $\pm ([b_+], [f_+])$. The manifolds $M_2$ and $M_3$ admit NPC metrics if and only if the entries of the matrices representing the gluing maps $\phi_i$, with respect to the chosen unitary framing, satisfy a system of quadratic equations. This condition is discussed in \cite{BuyaloSvetlov2005}. See also \cite{Leeb1995NPC3Manifolds} for a discussion of conditions under which linear graph manifolds admit NPC metrics.
\end{example}

\begin{figure}[ht]
    \centering
    \footnotesize
    
    \def\svgwidth{\columnwidth}
    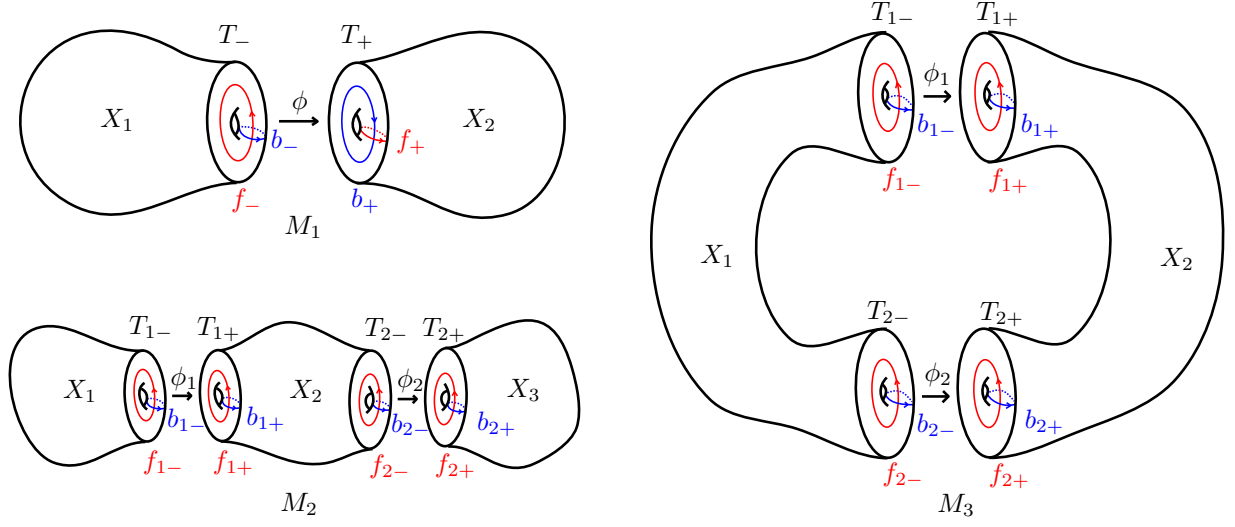

    \caption{Examples of graph manifolds and unitary framings $B_i$ for their JSJ pieces $X_i$. The red curves represent oriented regular fibers, and the blue curves represent the oriented curves in the unitary framings $B_i$.}
    \label{fig:npc_examples}
\end{figure}

\section{Proof of the Main Theorem} \label{sec:proof} 

In this section, we prove Theorem \ref{thm:npc_metric_free_action_invariant}. Let $M$ be a closed orientable graph manifold, and let $M' = M/G$ be the quotient of $M$ by a free, orientation-preserving diffeomorphism action of $G$. Denote the covering map by $p: M \rightarrow M'$. 

Let $\cT' = \{T'_1, \ldots, T'_{n}\}$ be a JSJ collection of incompressible tori in $M'$, and let $\Sigma' = \bigcup_{j=1}^n T'_j$ be their disjoint union. Since the action of $G$ on $M$ is free, the preimage of each torus $T'_i \in \cT'$ is a disjoint union of tori $T_{i,1}, \ldots, T_{i,n_i}$. By Lemma \ref{lem:JSJ_of_quotient_pulls_back_to_JSJ_of_original}, the collection
\begin{equation} \label{eqn:pullback_JSJ_tori}
    \cT = \{ T_{i,j} \subset p^{-1}(T'_i) \mid i = 1, \ldots, n; j = 1, \ldots, n_i \}
\end{equation}
is a collection of incompressible tori in $M$ such that every component of $M \mid \cT$ is Seifert fibered. Let $\Sigma = p^{-1}(\Sigma')$ be the union of the incompressible tori in $\cT$.

As discussed in $\S$\ref{subsec:roadmap}, if we can construct a $G$-invariant flat metric $\bar g$ on $\Sigma$ satisfying MEC on each component of $M \mid \cT$, then the induced flat metric $p_* \bar g$ on $\Sigma'$ also satisfies MEC on each component of $M' \mid \cT'$. We therefore first prove Theorem \ref{thm:npc_metric_free_action_invariant} under this assumption, which we state as the following proposition.

\medskip

\begin{proposition} \label{prop:npc_metric_on_M_equivariant_on_T}
    There exists a flat metric $g$ on $\Sigma$ that can be extended to an NPC metric on $M$, such that the metric
    \begin{equation} \label{eqn:sum_of_flat_metrics_is_flat}
        \bar g = \sum_{\phi \in G} \phi^* g
    \end{equation}
    on $\Sigma$ is flat, and for all $i = 1, \ldots, n$ and $j = 1, \ldots, n_i$, we have 
    \begin{equation} \label{eqn:sum_of_flat_metrics_is_flat 2}
        \sigma_{\bar g|_{T_{i,j}}} = \sum_{\phi \in G} \sigma_{(\phi^* g)|_{T_{i,j}}}.
    \end{equation}
    In particular, $\bar g$ extends to an NPC metric on $M$.
\end{proposition}

\begin{proof}[Proof of Theorem \ref{thm:npc_metric_free_action_invariant}]
    Recall from Remark \ref{rmk:generality_of_characteristic_coverings} that every finite cover of graph manifolds can be refined to a regular characteristic finite cover, so we may assume that $p: M \rightarrow M'$ is $s$-characteristic for some $s \in \ZZ_{>0}$. 
    Let $g$ be the flat metric on $\Sigma$ given by Proposition \ref{prop:npc_metric_on_M_equivariant_on_T}. Then $\bar g$ descends to a flat metric $g'$ on $\Sigma'$. We will show that $g'$ satisfies the conditions in Proposition \ref{prop:leeb_condition} for each component of $M' \mid \cT'$. 
    
    Let $X'$ be a connected component of $M' \mid \cT'$ whose preimage $X = p^{-1}(X')$ under the covering $p: M \rightarrow M'$ consists of components $X_1, \ldots, X_n$. The group $G$ acts transitively on the collection $\{X_1, \ldots, X_n\}$, so the degree of the covering $p|_{X_i}: X_i \rightarrow X'$ is independent of $i$; denote it by $d$. Since $p|_{X_i}$ is $s$-characteristic, the degree of the covering map restricted to a regular fiber is $s$. Also, $p|_{X_i}$ is regular, and the deck group $S_i = \Deck(X_i \rightarrow X')$ is the subgroup of $G$ whose elements $\phi \in S_i$ satisfy $\phi(X_i) = X_i$. In particular, $|S_i| = d$.

    Denote by $T'_1, \ldots, T'_{m'}$ the boundary components of $X'$. Their preimages are the tori $T_{i,j,k} \subset X_i$, where $p(T_{i,j,k}) = T'_j$ and $k = 1, \ldots, d/s^2$, as each $p^{-1}(T'_j)$ has $d/s^2$ many components, each covering $T'_j$ with degree $s^2$. Since $p$ is $s$-characteristic, there exist unitary framings $B_i = \bigcup_{j,k} b_{i,j,k}$ of $X_i$ and $B' = \bigcup_j b'_j$ of $X'$ which satisfy Proposition \ref{prop:charge_and_characteristic_covering}. Let $f_{i,j,k}$ and $f'_j$ be regular fibers on $T_{i,j,k}$ and $T'_j$ respectively. Then,
    \begin{equation*}
        p_* [f_{i,j,k}] = s [f'_j]
        \quad
        \text{and}
        \quad
        p_* [b_{i,j,k}] = s [b'_j].
    \end{equation*}

    Let $l$ be the $\bar g$-length of the regular fiber in $M$. 
    We first observe that
    \begin{equation*}
        \begin{aligned}
            s^2 \, \sigma_{g'|_{T'_j}} ([f'_j], [f'_j])
            &= \sigma_{g'|_{T_j}} (p_* [f_{i,j,k}], p_* [f_{i,j,k}])
            = \sigma_{(p^* g')|_{T_{i,j,k}}} ([f_{i,j,k}], [f_{i,j,k}]) \\
            &= \sigma_{\bar g|_{T_{i,j,k}}} ([f_{i,j,k}], [f_{i,j,k}]) = l^2.
        \end{aligned}
    \end{equation*}
    It follows that $\sigma_{g'|_{T'_j}} ([f'_j], [f'_j]) = l^2 / s^2$ and is independent of $j$.

    Next, applying Proposition \ref{prop:leeb_condition} to $(X_i, \bar g|_{\partial X_i})$ gives
    \begin{equation*}
        \begin{aligned}
            \ch(X_i, B_i) \, l^2
            &= \sum_{j=1}^{m'} \sum_{k=1}^{d/s^2} \sigma_{\bar g|_{T_{i,j,k}}} ([f_{i,j,k}], [b_{i,j,k}]) 
            = \sum_{j=1}^{m'} \sum_{k=1}^{d/s^2} \sigma_{g'|_{T'_j}} (p_* [f_{i,j,k}], p_* [b_{i,j,k}]) \\
            &= \sum_{j=1}^{m'} \frac{d}{s^2} \cdot \sigma_{g'|_{T'_j}} (s[f'_j], s[b'_j]) 
            = d \cdot \sum_{j=1}^{m'} \sigma_{g'|_{T'_j}} ([f'_j], [b'_j]).
        \end{aligned}
    \end{equation*}
    Also, Proposition \ref{prop:charge_and_characteristic_covering} gives us $\ch(X_i, B_i) = \frac{d}{s^2} \ch(X', B')$, so 
    \begin{equation*}
        \sum_{j=1}^{m'} \sigma_{g'|_{T'_j}} ([f'_j], [b'_j])
        = \ch(X_i, B_i) \, \frac{l^2}{d} 
        = \ch(X', B') \, \frac{l^2}{s^2}.
    \end{equation*}

    It follows that the framed Seifert fibered piece $(X', B')$ with the flat boundary metric $g'$ satisfy the conditions of Proposition \ref{prop:leeb_condition}. By Corollary \ref{cor:leeb_condition}, $g'$ extends to an NPC metric on $M'$. This completes the proof.
\end{proof}

\begin{figure}[ht]
    \centering
    \footnotesize
    
    \def\svgwidth{\columnwidth}
    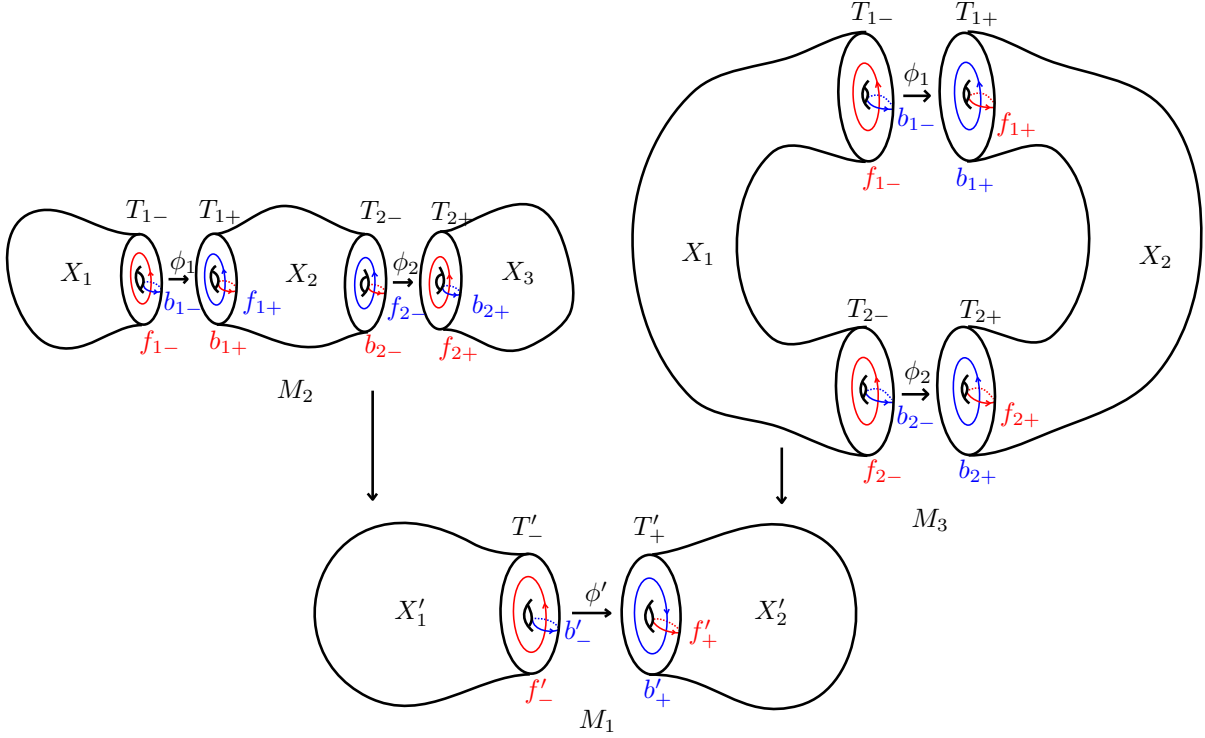

    \caption{Graph manifolds $M_2$ and $M_3$ admitting $\ZZ/2\ZZ$ symmetries whose quotient is $M_1$.}
    \label{fig:equivariant_npc_examples}
\end{figure}

\begin{example}
    Before proving Proposition \ref{prop:npc_metric_on_M_equivariant_on_T}, we give two examples illustrating how Theorem \ref{thm:npc_metric_free_action_invariant} applies; see Figure \ref{fig:equivariant_npc_examples}. The graph manifold $M_2$ in Figure \ref{fig:npc_examples} admits an involution preserving $X_2$ and swapping $X_1$ with $X_3$, provided that $X_1 \cong X_3$, that $X_2$ itself admits an involution swapping its two boundary components, and that the gluing maps $\phi_1$ and $\phi_2$ are compatible with these involutions. Similarly, the graph manifold $M_3$ in Figure \ref{fig:npc_examples} admits an involution if both $X_1$ and $X_2$ admit involutions swapping their two boundary components and if the gluing maps $\phi_1$ and $\phi_2$ are compatible with these involutions. In both cases, the quotient by the involution is $M_1$.
    By Theorem \ref{thm:npc_metric_free_action_invariant}, if $M_2$ or $M_3$ admits an NPC metric, then so does $M_1$. Therefore, the map $\phi'_*: H_1(T'_-) \rightarrow H_1(T'_+)$ must send $([f'_-], [b'_-])$ to $\pm ([b'_+], [f'_+])$. It follows that, for $i=1,2$, the maps $\phi_{i*}: H_1(T_{i-}) \rightarrow H_1(T_{i+})$ must send $([f_{i-}], [b_{i-}])$ to $\pm ([b_{i+}], [f_{i+}])$.
\end{example}

In the remaining part of this section, we will finish the proof of Theorem \ref{thm:npc_metric_free_action_invariant} by proving Proposition \ref{prop:npc_metric_on_M_equivariant_on_T} ($\S$\ref{subsec:proof_of_main_1}). Doing so requires us to state Proposition \ref{prop:isometric_finite_group_free_action_on_tori} regarding averaging flat metrics over a free finite group action on a disjoint collection of tori, which we prove in $\S$\ref{subsec:proof_of_main_2}.

\subsection{Proof of Proposition \ref{prop:npc_metric_on_M_equivariant_on_T}} \label{subsec:proof_of_main_1}

In this section, we prove Proposition \ref{prop:npc_metric_on_M_equivariant_on_T}. 
Our first step is to show that if $M$ admits an NPC metric, then it admits one for which the pullback tori $T_{i,j}$ from Equation \eqref{eqn:pullback_JSJ_tori} are totally geodesic and flat. For this we invoke a result of Leeb--Scott \cite{LeebScott1998CharacteristicAllDimensions}. The canonical decomposition in the theorem statement refers to a JSJ decomposition of $M$.

\medskip

\begin{theorem}[Leeb--Scott] \label{thm:geometric_JSJ}
    Let $M$ be a closed orientable $3$-manifold with an NPC metric. Then either $M$ has a flat metric, or $M$ can be canonically decomposed along finitely many totally geodesically embedded flat tori. The resulting pieces are Seifert fibered or atoroidal. The Seifert fibered components of $M$ admit a Seifert fibration by closed geodesics, and they are rigid in the sense that they split locally as a Riemannian product, with the fiber as the one-dimensional factor.
\end{theorem}

\medskip 

\begin{corollary} \label{cor:geometric_JSJ}
    Let $M$ and $\cT$ be defined as in Equation \eqref{eqn:pullback_JSJ_tori}. Suppose $M$ admits an NPC metric. Then there exists an NPC metric $h$ on $M$ such that every $T_{i,j} \in \cT$ is totally geodesic and flat with respect to $h$.
\end{corollary}

\begin{proof}
    Start with an NPC metric $\tilde h$ on $M$. By Theorem \ref{thm:geometric_JSJ}, there exists a JSJ collection of incompressible tori $\tilde\cT$ such that every $\tilde T \in \tilde\cT$ is totally geodesic and flat with respect to $\tilde h$. The components $\tilde X \subset M \mid \tilde\cT$ all satisfy MEC with respect to the metrics $\tilde h|_{\tilde T}$. 

    For each $T_{i,j} \in \cT$, there exists a unique $\tilde T_{i,j} \in \tilde T$ in the same isotopy class of $T_{i,j}$. Notice that if $\cT$ is not minimal, then it might be possible that $\tilde T_{i,j} = \tilde T_{i',j'}$ even when $(i,j) \neq (i',j')$. (In this case, $T_{i,j}$ and $T_{i',j'}$ together bound a $T^2 \times I$.) Fix a smooth isotopy 
    \begin{equation*}
        H_{i,j}: T^2 \times I \rightarrow M
    \end{equation*}
    satisfying $H_{i,j}(T^2 \times \{0\}) = T_{i,j}$ and $H_{i,j}(T^2 \times \{1\}) = \tilde T_{i,j}$. Let $f_{i,j}: T_{i,j} \rightarrow \tilde T_{i,j}$ be the diffeomorphism given by
    \begin{equation*}
        f_{i,j}(H_{i,j}(x, 0)) = H_{i,j}(x, 1) \quad \text{for all $x \in T^2$.}
    \end{equation*}
    Equip $T_{i,j}$ with the flat metric $h_{i,j} = f_{i,j}^* \tilde h|_{\tilde T_{i,j}}$. 

    We claim that the flat metrics $h_{i,j}$ on $T_{i,j}$ can be extended to every component of $M \mid \cT$ as in Proposition \ref{prop:leeb_condition}. Observe that $M \mid \cT$ differs from $M \mid \tilde \cT$ only by adding components diffeomorphic to $T^2 \times I$. 
    If $X \subset M \mid \cT$ is a component not diffeomorphic to $T^2 \times I$, then there is a component $\tilde X \subset M \mid \tilde \cT$ isotopic to $X$. In particular, the tori $T_{i,j} \subset \partial X$ correspond bijectively to the tori $\tilde T_{i,j} \subset \partial\tilde X$. Since $\tilde X$ satisfies MEC with respect to the metrics $\tilde h|_{\tilde T_{i,j}}$, $X$ must also satisfy MEC with respect to $h_{i,j}$.

    On the other hand, suppose $X \subset M \mid \cT$ is diffeomorphic to $T^2 \times I$. Let $T_{i,j}, T_{i',j'} \in \cT$ be the two boundary components of $X$, and let $\tilde T_{i,j} \in \tilde \cT$ be the unique JSJ torus in the isotopy class of both $T_{i,j}$ and $T_{i',j'}$.
    we may construct a diffeomorphism 
    \begin{equation*}
        F: X \rightarrow \tilde T_{i,j}^2 \times I.
    \end{equation*}
    such that 
    \begin{equation*}
        F^{-1}(x, t) = 
        \begin{cases}
            H_{i,j}(x,2t) & t \in [0, \frac{1}{2}] \\
            H_{i',j'}(x,1-2t) & t \in [\frac{1}{2}, 1]
        \end{cases}
    \end{equation*}
    Consider the product flat metric 
    \begin{equation*}
        h_X := \tilde h|_{\tilde T_{i,j}} + dt^2 \quad \text{over $\tilde T_{i,j} \times I$.}
    \end{equation*}
    Then, the metric $F^* h_X$ is flat on $X$ and agrees with $h_{i,j}$ and $h_{i',j'}$ on $\partial X = T_{i,j} \sqcup T_{i',j'}$. We conclude that the flat metrics $h_{i,j}$ can be extended to an NPC metric $h$ on $M$.
\end{proof}

Next, we state a proposition showing that we can in fact find an NPC metric on $M$ with the following additional property: not only are the pullback tori $T_{i,j}$ totally geodesic and flat, but the metric restricted to $\Sigma$ is also $G$-invariant. The idea is to use the ``averaging trick'' to first obtain a $G$-invariant metric $\bar g$ on $\Sigma$, and then show that $\bar g$ extends to an NPC metric on $M$. Since the proof of this proposition is long, we delay its proof to $\S$\ref{subsec:proof_of_main_2}.

\medskip

\begin{proposition} \label{prop:isometric_finite_group_free_action_on_tori}
    Let $\Sigma$ be a closed orientable flat surface, that is, a disjoint union of tori. Let $G$ be a finite group acting freely on $\Sigma$ by orientation-preserving diffeomorphisms, such that the quotient space $\Sigma / G$ is a torus. Then, for every flat metric $g$ on $\Sigma$, there exists $f \in \Diff_0(\Sigma)$ such that the metric
    \begin{equation} \label{eqn:sum_metric}
        \bar g = \sum_{\phi \in G} \phi^* (f^* g)
    \end{equation}
    is flat. In particular, $G$ acts isometrically on $\Sigma$ with respect to $\bar g$, and
    \begin{equation} \label{eqn:inner_product_and_sum_metric}
        \sigma_{\bar g|_{T_i}} = \sum_{\phi \in G} \sigma_{\phi^* f^* g|_{T_i}}
    \end{equation}
    for each torus component $T_i$ of $\Sigma$.
\end{proposition}

\begin{proof}[Proof of Proposition \ref{prop:npc_metric_on_M_equivariant_on_T}]
    Again, by Remark \ref{rmk:generality_of_characteristic_coverings}, we may without loss of generality assume that $M \rightarrow M'$ is $s$-characteristic. 
    Let $h$ be the NPC metric on $M$ obtained from Corollary \ref{cor:geometric_JSJ}. Since $\Sigma$ is preserved under the $G$-action, we may apply Proposition \ref{prop:isometric_finite_group_free_action_on_tori} to $(T, h|_T)$ to obtain a diffeomorphism $f \in \Diff_0(T)$. Let $H_f: T \times [0,1] \rightarrow T$ be an isotopy from $f$ to $\Id_T$. Let $N_T \subset M$ be a regular neighborhood of $\Sigma$, where $N_T \cong T \times [-1,1]$ and $T \subset N_T$ is identified with $T \times \{0\}$. Let $\tau: [0,1] \rightarrow [0,1]$ be a smooth nondecreasing function that is equal to $0$ in a neighborhood of $0$ and equal to $1$ in a neighborhood of $1$. Define a map $F: M \rightarrow M$ that is supported in $N_T$ and acts on $N_T \cong T \times [-1,1]$ by
    \begin{equation*}
        F(x, t) = (H_f(x, \tau(t^2)), t).
    \end{equation*}
    Observe that $F|_{T \times \{\pm 1\}} = \Id_{T \times \{\pm 1\}}$ and $F|_{T \times \{0\}} = f$. Also, $F$ is a diffeomorphism since it restricts to a diffeomorphism on each slice $T \times \{t\}$. Let $g = F^* h$. Then $g|_T = f^*(h|_T)$, so Equations \eqref{eqn:sum_of_flat_metrics_is_flat} and \eqref{eqn:sum_of_flat_metrics_is_flat 2} hold by Proposition \ref{prop:isometric_finite_group_free_action_on_tori}.

    We show that the $G$-invariant metric $\bar g$ from Equation \eqref{eqn:sum_of_flat_metrics_is_flat} extends to an NPC metric on $M$.
    Let $X_1$ be a Seifert piece in $M \mid \cT$. The orbit of $X_1$ consists of Seifert pieces $X_1, X_2, \ldots, X_k$. Let $T_{i,1}, \ldots, T_{i,m}$ be the boundary tori of $X_i$ (it is possible here that $T_{i,j}$ and $T_{i',j'}$ are isotopic in $M$, but that does not affect the proof). Denote by $X = X_1 \sqcup \cdots \sqcup X_k$, and let $X' = X / G$. Since $X$ admits an NPC metric, and the covering $X \rightarrow X'$ is regular and $s$-characteristic, we may invoke Proposition \ref{prop:charge_and_characteristic_covering} to obtain unitary framings $B = \bigcup_{i,j} b_{i,j}$ and $B' = \bigcup_i b'_i$ on $X$ and $X'$ respectively. In particular, $\phi_* [B] = [B]$ for all $\phi \in G$. Also, let $f_{i,j}$ and $f'_i$ be oriented regular fibers on $T_{i,j}$ and $T'_i$ respectively. 

    Denote by $B_i = B \cap X_i = \bigcup_j b_{i,j}$ the unitary framing restricted to each component $X_i$. 
    By Corollary \ref{cor:leeb_condition} and the fact that $G$ acts transitively on the collection of Seifert pieces $\{X_1, \ldots, X_k\}$, it suffices to show that $(X_1, B_1)$ satisfies MEC with respect to $\bar g|_{\partial X_1}$. Equation \eqref{eqn:sum_of_flat_metrics_is_flat 2} gives us 
    \begin{equation} \label{eqn:npc_metric_on_M_equivariant_on_T}
        \sigma_{\bar g|_{T_{1,j}}} = \sum_{\phi \in G} \sigma_{(\phi^* g)|_{T_{1,j}}}
        \quad
        \text{for all $j = 1, \ldots, m$}.
    \end{equation}
    We first show that $\sigma_{\bar g|_{T_{1,j}}}(f_{1,j}, f_{1,j})$ is independent of $j$. Denote by $l_i$ the $g$-length of a regular fiber on $X_i$. Let $S = \{\phi \in G \mid \phi(X_1) = X_1\}$ be the stabilizer subgroup of $X_1$ with respect to the permutation action on $\{X_1, \ldots, X_k\}$. Choose $\phi_1, \phi_2, \ldots, \phi_k \in G$ such that $\phi_i(X_1) = X_i$ and $\phi_i(T_{1,j}) = T_{i,j}$ for all $j = 1, \ldots, m$. Then every element $\phi \in G$ can be uniquely written as $\phi_i \psi$ for some $\psi \in S$. 
    Equation \eqref{eqn:npc_metric_on_M_equivariant_on_T} gives us 
    \begin{equation*}
        \begin{aligned}
            \sigma_{\bar g|_{T_{1,j}}}([f_{1,j}], [f_{1,j}])
            &= \sum_{\phi \in G} \sigma_{(\phi^* g)|_{T_{1,j}}} ([f_{1,j}], [f_{1,j}]) 
            = \sum_{\phi \in G} \sigma_{g|_{\phi(T_{1,j})}} (\phi_* [f_{1,j}], \phi_* [f_{1,j}]) \\
            &= \sum_{i=1}^k \sum_{\psi \in S} \sigma_{g|_{T_{i,j}}} ((\phi_i \psi)_* [f_{1,j}], (\phi_i \psi)_* [f_{1,j}]).
        \end{aligned}
    \end{equation*}
    Since $\psi \in S$ fixes $X_1$, it permutes the boundary tori $T_{1,1}, \ldots, T_{1,m}$, so our choice of unitary framing from Proposition \ref{prop:charge_and_characteristic_covering} implies that $\psi$ permutes the pairs $([f_{1,j}], [b_{1,j}])$. Let $\psi(j)$ be the index such that $\psi(T_{1,j}) = T_{1,\psi(j)}$. It follows that
    \begin{equation*}
        \sigma_{\bar g|_{T_{1,j}}}([f_{1,j}], [f_{1,j}])
        = \sum_{i=1}^k \sum_{\psi \in S} \sigma_{g|_{T_{i,j}}} ([f_{i,\psi(j)}], [f_{i,\psi(j)}]) 
        = |S| \cdot \sum_{i=1}^k l_i^2.
    \end{equation*}
    Therefore, setting $l^2 = |S| \cdot \sum_{i=1}^k l_i^2$, we see that $\sigma_{\bar g|_{T_{1,j}}}([f_{1,j}], [f_{1,j}]) = l^2$ is independent of $j$. 

    Next, we show that Equation \eqref{eqn:leeb_condition} holds for $(X_1, B_1)$. Equation \eqref{eqn:npc_metric_on_M_equivariant_on_T} gives us
    \begin{equation*}
        \begin{aligned}
            \sum_{j=1}^m \sigma_{\bar g|_{T_{1,j}}}([f_{1,j}], [b_{1,j}])
            &= \sum_{\phi \in G} \sum_{j=1}^m \sigma_{(\phi^* g)|_{T_{1,j}}} ([f_{1,j}], [b_{1,j}]) \\
            &= \sum_{i=1}^k \sum_{\psi \in S} \sum_{j=1}^m \sigma_{(\phi_{i}\psi)^* g|_{T_{1,j}}} ([f_{1,j}], [b_{1,j}]) \\
            &= \sum_{i=1}^k \sum_{\psi \in S} \sum_{j=1}^m \sigma_{g|_{T_{i,j}}} ((\phi_{i}\psi)_* [f_{1,j}], (\phi_{i}\psi)_* [b_{1,j}]) .
        \end{aligned}
    \end{equation*}
    Again, since $\psi$ permutes the pairs $([f_{1,j}], [b_{1,j}])$, we get 
    \begin{equation*}
        \begin{aligned}
            \sum_{i=1}^k \sum_{\psi \in S} \sum_{j=1}^m \sigma_{g|_{T_{i,j}}} ((\phi_{i}\psi)_* [f_{1,j}], (\phi_{i}\psi)_* [b_{1,j}])
            &= \sum_{i=1}^k \sum_{\psi \in S} \sum_{j=1}^m \sigma_{g|_{T_{i,j}}} ((\phi_{i})_* [f_{1,\psi(j)}], (\phi_{i})_* [b_{1,\psi(j)}]) \\
            &= |S| \sum_{i=1}^k \sum_{j=1}^m \sigma_{g|_{T_{i,j}}} ((\phi_{i})_* [f_{1,j}], (\phi_{i})_* [b_{1,j}]).
        \end{aligned}
    \end{equation*}
    Finally, since $\phi_i(T_{1,j}) = T_{i,j}$, we have
    \begin{equation*}
        |S| \sum_{i=1}^k \sum_{j=1}^m \sigma_{g|_{T_{i,j}}} ((\phi_{i})_* [f_{1,j}], (\phi_{i})_* [b_{1,j}])
        = |S| \sum_{i=1}^k \sum_{j=1}^m \sigma_{g|_{T_{i,j}}} ([f_{i,j}], [b_{i,j}]) .
    \end{equation*}
    Now, for each $i$, the flat metrics $\{g|_{T_{i,j}}\}_{j=1}^m$ extend to an NPC metric on $X_i$, Proposition \ref{prop:leeb_condition} gives
    \begin{equation*}
        |S| \sum_{i=1}^k \sum_{j=1}^m \sigma_{g|_{T_{i,j}}} ([f_{i,j}], [b_{i,j}])
        = |S| \sum_{i=1}^k (\ch(X_i, B_i) \, l_i^2). 
    \end{equation*}
    Since $B$ is $G$-invariant, for all $i, i'$, the pair $(X_i, B_i)$ and $(X_{i'}, B_{i'})$ are isomorphic as framed Seifert fibered spaces, so $\ch(X_i, B_i)$ is independent of $i$. It follows that 
    \begin{equation*}
        \sum_{j=1}^m \sigma_{\bar g|_{T_{1,j}}}([f_{1,j}], [b_{1,j}])
        = \ch(X_1, B_1) \cdot \left( |S| \cdot \sum_{i=1}^k  l_i^2 \right)
        = \ch(X_1, B_1) \, l^2.
    \end{equation*}
    Therefore, Equation \eqref{eqn:leeb_condition} holds. 
    This completes the proof that $\bar g|_{\partial X_1}$ satisfies the conditions in Proposition \ref{prop:leeb_condition}.
\end{proof}

\subsection{Proof of Proposition \ref{prop:isometric_finite_group_free_action_on_tori}} \label{subsec:proof_of_main_2}

The goal of this section is to prove Proposition \ref{prop:isometric_finite_group_free_action_on_tori}, which requires us to work with properties of translation subgroups of the diffeomorphism group of a $2$-dimensional torus. 
Denote by $T^2$ the $2$-dimensional torus. Let $\TT^2 \cong \SO(2)^2$ be the rank-$2$ compact abelian Lie group. There is a homomorphism
\begin{equation*}
    \alpha: \GL(2, \ZZ) \hookrightarrow \Aut(\TT^2)
\end{equation*}
defined by
\begin{equation*}
    \alpha\!\left(
        \begin{pmatrix}
        a & b\\
        c & d
        \end{pmatrix}
    \right)
    :
    (z, w)
    \mapsto 
    (z^a w^b, z^c w^d)
\end{equation*}
Here, we treat $\SO(2)$ as the unit complex numbers.

A \emph{torus subgroup} of $\Diff_0(T^2)$ is an equivalence class of faithful representations
\begin{equation*}
    \rho: \TT^2 \hookrightarrow \Diff_0(T^2),
\end{equation*}
where two representations $\rho_1$ and $\rho_2$ are declared equivalent if there exists $A \in \GL(2, \ZZ)$ such that 
\begin{equation*}
    \rho_2 = \rho_1 \circ \alpha(A).
\end{equation*}
Thus, we quotient by reparametrizations of $\TT^2$ so that a torus subgroup records only the image $\rho(\TT^2) \subset \Diff_0(T^2)$ and not the particular marking of $\TT^2$.

We start by showing that faithful actions of $\TT^2$ on $T^2$ are free and transitive. To do so, we need the following lemma, which follows from naturality of the Riemannian metric exponential map with isometries. For a proof, see \cite[Proposition 5.22]{lee2019riemannian}.

\medskip

\begin{lemma} \label{lem:isometry_fixed_point_trivial_differential}
    Let $(M, g)$ be a connected Riemannian manifold, and let $\phi: M \rightarrow M$ be an isometry such that there exists $x \in M$ with $\phi(x) = x$ and $d\phi_x = \Id_{T_xM}$. Then, $\phi = \Id_M$. 
\end{lemma}

\medskip

\begin{corollary} \label{cor:isometry_fixed_point_trivial_differential}
    Let $(M, g)$ be a connected Riemannian manifold, and let $G$ be a Lie group acting isometrically on $M$. If there exists a point $x \in M$ such that for all $\phi \in G$, $\phi(x) = x$ and $d\phi_x = \Id_{T_xM}$, then, $G$ acts trivially on $M$. 
\end{corollary}

\begin{proof}
    Apply Lemma \ref{lem:isometry_fixed_point_trivial_differential} to every $\phi \in G$ to get $\phi = \Id_M$. 
\end{proof}

\medskip

\begin{lemma} \label{lem:faithful_torus_action_is_free_transitive}
    Every faithful smooth action of $\TT^2$ on $T^2$ is free and transitive. 
\end{lemma}

\begin{proof}
    We first show that any faithful action of $\TT^2$ on $T^2$ is free. We start by observing that there exists a Riemannian metric $g$ on $T^2$ such that $\TT^2$ acts isometrically: one may take $g$ to be the average of an arbitrary Riemannian metric on $T^2$ by the Haar measure of $\TT^2$. Therefore, we may without loss of generality assume that $\TT^2$ acts isometrically on $(T^2, g)$. 
    
    For all $x \in T^2$, let $S_x < \TT^2$ be the stabilizer subgroup of $x$. Note that $S_x$ is always a closed Lie subgroup of $\TT^2$. Denote by $\Orb(x) \subset T^2$ the orbit of $x$ under the action of $\TT^2$. By the orbit-stabilizer theorem for smooth Lie group actions, we have $\Orb(x) \cong \TT^2 / S_x$. In particular, since $\TT^2$ is abelian, $S_x$ acts trivially on $\Orb(x)$. We show that $S_x = \{e\}$ and $\Orb(x) = T^2$. 

    We first consider the case where $\dim S_x = 2$. In this case, $S_x$ is an open Lie subgroup of $\TT^2$. Since $\TT^2$ is connected, we must have $S_x = \TT^2$. It follows that $\TT^2$ fixes $x$. Also, the image of $S_x$ under the differential map $d_x: S_x \rightarrow \GL(T_xT^2)$ given by 
    \begin{equation} \label{eqn:differential_map}
        d_x(\phi) = \left( d\phi_x: T_xT^2 \rightarrow T_xT^2 \right) \in \GL(T_xT^2)
    \end{equation}
    must be contained in a maximal compact subgroup of $\GL(T_xT^2) \cong \GL(2, \RR)$, which is 1-dimensional. It follows that $\ker(d_x) < \TT^2$ is nontrivial, so by Corollary \ref{cor:isometry_fixed_point_trivial_differential}, $\ker(d_x)$ acts trivially on $T^2$, which contradicts faithfulness of the $\TT^2$ action. 

    Next, we consider the case where $\dim S_x = 1$. In this case, we have $S_x \cong \SO(2) \times \ZZ/n\ZZ$. By the orbit-stabilizer theorem for smooth Lie group actions, we get that the orbit $\Orb(x) \cong \TT^2 / S_x \cong S^1$, so $\Orb(x)$ is a smoothly embedded simple closed curve in $T^2$. Now, let $d_x: S_x \rightarrow \GL(T_xT^2)$ be the differential map as in Equation \eqref{eqn:differential_map}. Since $S_x$ acts trivially on $\Orb(x)$, each element of $S_x$ fixes the tangent line $T_x\Orb(x)$ pointwise. By picking a suitable basis for $T_xT^2$, we get that 
    \begin{equation*}
        d_x(\phi) = 
        \begin{pmatrix}
            1 & f_1(\phi) \\ 0 & f_2(\phi)
        \end{pmatrix}
    \end{equation*}
    where $f_1(\phi) \in \RR$ and $f_2(\phi) \in \RR^{\times}$. It follows that $d_x(S_x)$ lies in a Lie subgroup of $\GL(T_xT^2)$ that is isomorphic to $\RR \rtimes \RR^{\times}$. Since the maximal compact subgroup of $\RR \rtimes \RR^{\times}$ is finite, the kernel $K = \ker(d_x) < S_x < \TT^2$ is a nontrivial subgroup and acts trivially on $T^2$ by Corollary \ref{cor:isometry_fixed_point_trivial_differential}, which contradicts faithfulness of the $\TT^2$ action. 

    Finally, we consider the case where $\dim S_x = 0$. In this case, $S_x$ is a finite subgroup of $\TT^2$, so $S_x \cong \ZZ/m\ZZ \times \ZZ/n\ZZ$ for some positive integers $m,n$. Again, by the orbit-stabilizer theorem for smooth Lie group actions, we get $\Orb(x) \cong \TT^2 / S_x \cong T^2$. We claim that $\Orb(x) = T^2$. Denote by $\Phi_x: \TT^2 \rightarrow T^2$ the orbit map sending 
    \begin{equation*}
        \Phi_x(\phi) = \phi(x).
    \end{equation*}
    Observe that $\Phi_x(\TT^2) = \Orb(x)$. Notice that $\Phi_x$ is closed, since the domain is compact and the codomain is Hausdorff. To see that $\Phi_x$ is open, we show that it is a local diffeomorphism. It suffices to show that $\Phi_x$ is a local diffeomorphism at $e$. Notice that $\ker d(\Phi_x)_e$ is precisely the Lie algebra $\ks_x$ of $S_x$, and from $\dim S_x = 0$ we get that $\dim \ker d(\Phi_x)_e = \dim \ks_x = 0$, $d(\Phi_x)_e$ is full-rank. 
    It follows that $\Phi_x$ is both a closed map and an open map, so$\Orb(x) = \Phi_x(\TT^2)$ is clopen in $T^2$, which gives us $\Orb(x) = T^2$. Finally, since every element of $S_x$ acts trivially on $\Orb(x) = T^2$, the faithfulness of the $\TT^2$ action gives us $S_x = \{e\}$. We conclude that $\TT^2$ must act freely and transitively. 
\end{proof}

Next, we prove the following properties of torus subgroups, which will become useful in the proof of Proposition \ref{prop:isometric_finite_group_free_action_on_tori}.

\medskip

\begin{lemma} \label{lem:properties_of_diff_of_torus}
    The following statements are true:
    \begin{enumerate}
        \item[\textnormal{(a)}] Every pair of torus subgroups of $\Diff_0(T^2)$ is conjugate in $\Diff_0(T^2)$. That is, given two torus subgroups $[\rho_1], [\rho_2]$, there exist representatives $\rho_1 \in [\rho_1]$, $\rho_2 \in [\rho_2]$, and $f \in \Diff_0(T^2)$ such that
        \begin{equation*}
            \rho_2(t) = f \circ \rho_1(t) \circ f^{-1} \quad \text{for all $t \in \TT^2$.}
        \end{equation*}

        \item[\textnormal{(b)}] Every finite subgroup of $\Diff_0(T^2)$ factors through a torus subgroup. That is, given any representation $\phi: G \hookrightarrow \Diff_0(T^2)$ of a finite group $G$, there exist $\phi': G \hookrightarrow \TT^2$ and $\rho: \TT^2 \hookrightarrow \Diff_0(T^2)$ such that the following diagram commutes:
        \[\begin{tikzcd}
            G && {\Diff_0(T^2)} \\
            & \TT^2
            \arrow["\phi", from=1-1, to=1-3]
            \arrow["{\phi'}"', from=1-1, to=2-2]
            \arrow["\rho"', from=2-2, to=1-3]
        \end{tikzcd}\]

        \item[\textnormal{(c)}] Let $g$ be a Riemannian metric on $T^2$. Then $g$ is flat if and only if there exists a torus subgroup $\rho: \TT^2 \hookrightarrow \Diff_0(T^2)$ that factors through $\Isom(T^2, g)$, meaning that there exists a map $\TT^2 \hookrightarrow \Isom(T^2, g)$ such that the following diagram commutes:
        % https://q.uiver.app/#q=WzAsNCxbMCwwLCJcXFRUIl0sWzEsMCwiXFxEaWZmXzAoVF4yKSJdLFsxLDEsIlxcRGlmZihUXjIpIl0sWzAsMSwiXFxJc29tKFReMiwgZykiXSxbMCwxLCJcXHJobyJdLFsxLDJdLFswLDMsIlxcZXhpc3RzIiwyXSxbMywyXV0=
        \[\begin{tikzcd}
            \TT^2 & {\Diff_0(T^2)} \\
            {\Isom(T^2, g)} & {\Diff(T^2)}
            \arrow["\rho", from=1-1, to=1-2]
            \arrow["\exists"', from=1-1, to=2-1]
            \arrow[from=1-2, to=2-2]
            \arrow[from=2-1, to=2-2]
        \end{tikzcd}\]
        Here, $\Diff_0(T^2) \rightarrow \Diff(T^2)$ is the natural inclusion, and $\Isom(T^2, g) \rightarrow \Diff(T^2)$ is the natural forgetful map.
        In particular, if $\rho$ factors through $\Isom(T^2, g)$, then $\rho \circ \alpha(A)$ also factors through $\Isom(T^2, g)$ for all $A \in \GL(2, \ZZ)$.
    \end{enumerate}
\end{lemma}

\begin{proof}
    For (a), we start with arbitrary representatives $\rho_i: \TT^2 \hookrightarrow \Diff_0(T^2)$ of the torus subgroups $[\rho_i]$ for $i = 1, 2$. Fix a point $p \in T^2$. Since each $\rho_i$ is a faithful representation, each action of $\TT^2$ on $T^2$ induced by $\rho_i$ is free and transitive. This implies the orbit map 
    \begin{equation*}
        \Phi_i : \TT^2 \rightarrow T^2,
        \qquad
        \Phi_i(g) = \rho_i(g)(p)
    \end{equation*}
    are diffeomorphims. Fix a basis $(a, b)$ of $H_1(T^2)$, and let $(a', b')$ be a basis of $H_1(\TT^2)$ generated by the coordinate directions of $\SO(2)^2$. Let $A_i \in \GL(2, \ZZ)$ be the matrix of the induced homomorphism $(\Phi_i)_*: H_1(\TT^2) \rightarrow H_1(T^2)$. Set 
    \begin{equation*}
        A = A_1^{-1} A_2 \in \GL(2, \ZZ)
        \quad \text{and} \quad
        f = \Phi_2 \circ \alpha(A^{-1}) \circ \Phi_1^{-1}.
    \end{equation*}
    By construction, $(\Phi_2)_* \circ A^{-1} \circ (\Phi_1)_*^{-1} = \Id$, so $f_* = \Id$ on $H_1(T^2)$, hence $f \in \Diff_0(T^2)$. 
    
    It then remains to show that for all $g \in \TT^2$, we have
    \begin{equation*}
        f \circ (\rho_1 \circ \alpha(A))(g) \circ f^{-1} = \rho_2(g).
    \end{equation*}
    Given any $q \in T^2$, we let $g' \in \TT^2$ satisfy $\rho_2(g')(p) = q$, so $\Phi_2^{-1}(q) = g'$. This gives us
    \begin{equation*}
        f^{-1}(q) 
        = (\Phi_1 \circ \alpha(A) \circ \Phi_2^{-1}) (q)
        = \Phi_1 (\alpha(A)(g')) 
        = (\rho_1 \circ \alpha(A)) (g') (p).
    \end{equation*}
    It follows that 
    \begin{equation*}
        \begin{aligned}
            (f \circ (\rho_1 \circ \alpha(A))(g) \circ f^{-1}) (q)
            &= (f \circ (\rho_1 \circ \alpha(A))(g + g')) (p)
            = \Phi_2 \circ \alpha(A^{-1}) (\alpha(A)(g + g')) \\
            &= \Phi_2(g + g')
            = \rho_2(g + g') (p)
            = \rho_2(g)(q).
        \end{aligned}
    \end{equation*}
    This completes the proof of (a).

    For (b), since $G$ is a finite subgroup of $\Diff_0(T^2)$, it preserves the orientation of $T^2$. Moreover, $G$ acts trivially on $H_1(T^2)$, so every nontrivial element of $G$ acts freely on $T^2$. Hence the quotient $T^2/G$ is again homeomorphic to $T^2$.

    We equip the quotient $T^2/G$ with a Lie group structure $\TT'$. This Lie group structure induces a Lie group structure $\TT$ on $T^2$ such that the quotient map $T^2 \rightarrow T^2/G$ is a Lie group homomorphism $\TT \rightarrow \TT'$. It follows that the deck transformations $\Deck(T^2 \rightarrow T^2/G) \cong G$ act by left multiplication with respect to the Lie group structure $\TT$ on $T^2$.

    Now consider the homomorphism $\rho: \TT \rightarrow \Diff_0(T^2)$ given by left multiplication $\rho(g)(h) = g \cdot h$. This is a faithful representation, and the preceding discussion shows that the action of $G$ factors through $\rho$, as desired.

    For (c), notice that if $g$ is a flat metric on $T^2$, then $\Isom(T^2, g)$ always contain a translation subgroup isomorphic to $\TT^2$. On the other hand, if $g$ is a left-invariant metric on $T^2$ under some transitive Lie group action, then by the Gauss--Bonnet theorem, $g$ must be a flat metric. The last remark of (c) follows from the fact that $(\rho \circ \alpha(A))(\TT^2) = \rho(\TT^2)$ for all $A \in \GL(2, \ZZ)$.
\end{proof}

\medskip

\begin{remark}
    Lemma \ref{lem:faithful_torus_action_is_free_transitive} could be avoided by defining torus subgroups to act freely and transitively. With this stronger condition in the definition, all parts of Lemma \ref{lem:properties_of_diff_of_torus} would still hold. We instead choose to only assume faithfulness and provide a self-contained proof of Lemma \ref{lem:faithful_torus_action_is_free_transitive}.
\end{remark}

\medskip

\begin{corollary} \label{cor:properties_of_diff_of_torus}
    If $g_1, \ldots, g_n$ are flat metrics on $T^2$, and $\rho: \TT^2 \hookrightarrow \Diff_0(T^2)$ factors through $\Isom(T^2, g_i)$ for every $i = 1, \ldots, n$, then $\rho$ factors through $\Isom(T^2, \bar g)$, where
    \begin{equation*}
        \bar g = \sum_{i=1}^n g_i.
    \end{equation*}
    In particular, the metric $\bar g$ is flat on $T^2$.
\end{corollary}

\begin{proof}
    For all $\phi \in \rho(\TT^2)$, we know that $\phi^* g_i = g_i$. This in particular implies
    \begin{equation*}
        \phi^* \bar g 
        = \phi^* \left( \sum_{i=1}^n g_i \right)
        = \sum_{i=1}^n \phi^* g_i 
        = \sum_{i=1}^n g_i 
        = \bar g.
    \end{equation*}
    Therefore, $\rho: \TT^2 \hookrightarrow \Diff_0(T^2)$ factors through $\Isom(T^2, \bar g)$. By Lemma \ref{lem:properties_of_diff_of_torus}(c), $\bar g$ is flat.
\end{proof}

We now proceed to prove Proposition \ref{prop:isometric_finite_group_free_action_on_tori}. The central idea behind the proof is the following. Suppose $X = \{x_1, \ldots, x_n\}$ is a finite set on which a finite group $G$ acts transitively, and suppose one would like to assign a structure $a_i$ to each $x_i \in X$ in a way that is compatible with the group action, but assigning all $a_i$ simultaneously is difficult. One begins by choosing a structure $a_1$ on $x_1$ that is compatible with the action of the stabilizer subgroup $S < G$ of $x_1$. Then, by the orbit--stabilizer theorem, $G$ can be decomposed into cosets $\phi_1 S, \ldots, \phi_n S$, where each $\phi_i \in G$ sends $x_1$ to $x_i$. One then defines a structure $a_i$ on $x_i$ by requiring that $\phi_i^* a_i$ and $a_1$ are compatible. The resulting collection $a_1, \ldots, a_n$ is compatible with the action of the whole group $G$ by purely group-theoretic considerations. Nevertheless, this construction is not ``natural'' in the sense that it depends on the choice of the basepoint $x_1$ and the choice of representatives $\phi_i$.

In the proof of Proposition \ref{prop:isometric_finite_group_free_action_on_tori}, the set $X$ will be the collection of torus components of $\Sigma$, the group $G$ will be a finite group acting freely on $\Sigma$, which acts transitively on $X$ since $\Sigma/G$ is connected, and the structures $a_i$ will be diffeomorphisms $f_i \in \Diff_0(T_i)$ such that the torus subgroups of the isometry groups of $f_i^* g|_{T_i}$ and $f_j^* g|_{T_j}$ are related by the $G$-action.

\begin{proof}[Proof of Proposition \ref{prop:isometric_finite_group_free_action_on_tori}]
    We write $\Sigma = T_1 \sqcup \cdots \sqcup T_n$. Consider the subgroup $S < G$ given by
    \begin{equation*}
        S := \{ \phi \in G \, | \, \phi(T_1) = T_1 \}.
    \end{equation*}
    Since $\Sigma / G$ is a single torus, we see that $G$ acts transitively on the collection of components $\{T_1, \ldots, T_n\}$, so we can choose $\phi_i \in G$ for each $i = 1, \ldots, n$ such that $\phi_i(T_1) = T_i$ (in particular, we may pick $\phi_1 = \Id$), and the left $S$-cosets of $G$ are precisely
    \begin{equation*}
        \phi_1S = S, 
        \phi_2 S, 
        \ldots, 
        \phi_n S.
    \end{equation*}
    Also, since $g|_{T_i}$ is flat on each torus $T_i$, Lemma \ref{lem:properties_of_diff_of_torus}(c) implies that for each $i = 1, \ldots, n$, there exists a torus subgroup
    \begin{equation*}
        \tau_i: \TT^2 \hookrightarrow \Diff_0(T_i)
    \end{equation*}
    that factors through $\Isom(T_i, g|_{T_i})$.

    We now consider the action of $S$ on $T_1$. Since $S$ is finite, Lemma \ref{lem:properties_of_diff_of_torus}(b) implies that $S \rightarrow \Diff_0(T_1)$ factors through a map $\rho_1: \TT^2 \hookrightarrow \Diff_0(T_1)$ as
    \begin{equation*}
        S \rightarrow \TT^2 \xrightarrow{\rho_1} \Diff_0(T_1).
    \end{equation*}
    By Lemma \ref{lem:properties_of_diff_of_torus}(a), after possibly replacing $\tau_1$ with $\tau_1 \circ \alpha(A)$ for some $A \in \GL(2, \ZZ)$, there exists $f_1 \in \Diff_0(T_1)$ such that
    \begin{equation*}
        \rho_1 = f_1 \circ \tau_1 \circ f_1^{-1}.
    \end{equation*}
    This in particular implies that $\rho_1$ factors through $\Isom(T_1, f_1^*(g|_{T_1}))$, so the action of $S$ restricted to $T_1$ is isometric with respect to the flat metric $f_1^* (g|_{T_1})$.

    Next, we take care of the remaining pieces $T_i$. Consider the torus subgroup
    \begin{equation*}
        \rho_i := \phi_i \circ \rho_1 \circ \phi_i^{-1}: \TT^2 \hookrightarrow \Diff_0(T_i).
    \end{equation*}
    Again, by Lemma \ref{lem:properties_of_diff_of_torus}(a), after replacing $\tau_i$ by $\tau_i \circ \alpha(A_i)$ for some $A_i \in \GL(2, \ZZ)$ if necessary, we can find $f_i \in \Diff_0(T_i)$ such that
    \begin{equation*}
        \rho_i = f_i \circ \tau_i \circ f_i^{-1}.
    \end{equation*}
    In particular, the torus subgroup $\rho_i$ factors through $\Isom(T_i, f_i^*(g|_{T_i}))$. Equivalently, $\rho_1$ factors through
    \begin{equation*}
        \Isom(T_1, \phi_i^* f_i^* (g|_{T_i})) \quad \text{for all $i = 1, \ldots, n$,}
    \end{equation*}
    so Lemma \ref{lem:properties_of_diff_of_torus}(b) implies that $S$ acts by isometries on $\phi_i^* f_i^* (g|_{T_i})$.

    We now define $f \in \Diff_0(\Sigma)$ by
    \begin{equation*}
        f|_{T_i} = f_i \quad \text{for all $i = 1, \ldots, n$.}
    \end{equation*}
    We show that $f$ is the desired map in Proposition \ref{prop:isometric_finite_group_free_action_on_tori}. We first show that the torus subgroup $\rho_1: \TT^2 \hookrightarrow \Diff_0(T_1)$ factors through
    \begin{equation*}
        \Isom(T_1, (\phi^* f^* g)|_{T_1}) \quad \text{for all $\phi \in G$.}
    \end{equation*}
    Write $\phi = \phi_i \circ s$ for some $i = 1, \ldots, n$ and $s \in S$. Since $S$ acts on $(\phi_i^*f^*g)|_{T_1} = \phi_i^*f_i^*(g|_{T_i})$ by isometries, we have
    \begin{equation*}
        (\phi^* f^*g)|_{T_1}
        = (s^* \phi_i^* f^* g)|_{T_1}
        = (\phi_i^* f^* g)|_{T_1}
        = \phi_i^* f_i^* (g|_{T_i}).
    \end{equation*}
    Since we have already shown that $\rho_1$ factors through $\Isom(T_1, \phi_i^* f_i^* (g|_{T_i}))$, it must also factor through $\Isom(T_1, (\phi^* f^* g)|_{T_1})$. In particular, this implies that $\rho_i = \phi_i \circ \rho_1 \circ \phi_i^{-1}$ factors through
    \begin{equation*}
        \Isom(T_i, ((\phi_i^{-1})^* \phi^* f^* g)|_{T_i}) \quad \text{for all $\phi \in G$ and $i = 1, \ldots, n$.}
    \end{equation*}
    Since every $\phi \in G$ can be written as $\psi \circ \phi_i^{-1}$ for some $\psi \in G$, we see that for all $i = 1, \ldots, n$ and $\phi \in G$, the torus subgroup $\rho_i$ factors through $\Isom(T_i, (\phi^* f^* g)|_{T_i})$. Applying Corollary \ref{cor:properties_of_diff_of_torus} to each $T_i$ with the flat metrics $(\phi^* f^* g)|_{T_i}$ for each $\phi \in G$ and the torus subgroup $\rho_i: \TT^2 \hookrightarrow \Diff_0(T_i)$, we see that the metric $\bar g|_{T_i}$ from Equation \eqref{eqn:sum_metric} is flat.

    Finally, Equation \eqref{eqn:inner_product_and_sum_metric} holds because $f|_{T_i} = f_i$ is chosen from $\Diff_0(T_i)$, and pulling back by $\Diff_0(T_i)$ does not change the metric-induced inner product on $H^1(T_i)$.
\end{proof}

\printbibliography

\end{document}